\documentclass [11pt,reqno]{amsart}

\usepackage{amsmath,amsfonts,amssymb,amscd,verbatim}

\def\BQ{{\mathbb Q}}
\def\BZ{{\mathbb Z}}
\def\BC{{\mathbb C}}
\def\BR{{\mathbb R}}
\def\BF{{\mathbb F}}

\def\ib{{\mathbf i}}

\def\Gal{\mathrm{Gal}}
\def\Perm{\mathrm{Perm}}

\def\tr{\mathrm{tr}}
\def\so{\frak{so}}
                                   \def\HH{\mathrm{H}}

\def\Hdg{\mathrm{Hdg}}

\def\End{\mathrm{End}}
\def\Aut{\mathrm{Aut}}
\def\Hom{\mathrm{Hom}}

\def\Mat{\mathrm{Mat}}

\def\fchar{\mathrm{char}}

\def\GL{\mathrm{GL}}
\def\SL{\mathrm{SL}}

\def\Sp{\mathrm{Sp}}

\def\dim{\mathrm{dim}}

                          \def\RR{{\mathfrak R}}

                    \def\sll{\mathfrak{sl}}

\newtheorem{thm}{Theorem}[section]
\newtheorem{lem}[thm]{Lemma}
\newtheorem{cor}[thm]{Corollary}
\newtheorem{prop}[thm]{Proposition}

\theoremstyle{definition}
\newtheorem{defn}[thm]{Definition}
\newtheorem{ex}[thm]{Example}

\newtheorem{sect}[thm]{}

           \newtheorem{rem}[thm]{Remark}

\newtheorem{rems}[thm]{Remarks}

\begin{document}

\title[2-simple Complex Tori] {Endomorphism Algebras and Automorphism Groups of certain Complex Tori}

\author{Yuri  G. Zarhin}
\address{Pennsylvania State University, Department of Mathematics, University Park, PA 16802, USA}
\email{zarhin@math.psu.edu}
\thanks{The  author was partially supported by Simons Foundation Collaboration grant   \# 585711
 and the Travel Support for Mathematicians Grant MPS-TSM-00007756 from the Simons Foundation.
Most  of this work was done in January--May 2022 and December 2023 during his stay at the Max-Planck Institut f\"ur Mathematik (Bonn, Germany), whose hospitality and support are gratefully acknowledged.}

\begin{abstract}
We study the endomorphism algebra, the automorphism  group and the Hodge group of  complex tori, whose second rational cohomology group enjoys a certain Hodge
property introduced by
 F. Campana.

 \end{abstract}

\maketitle
 \tableofcontents

\section{Introduction}
As usual, we write $\BQ,\BR,\BC$ for the fields of rational, real, complex numbers and $\BZ$ for the ring of integers.
We write $\bar{\BQ}$ for the subfield of all algebraic nunbers in $\BC$, which is an algebraic closure of $\BQ$.
If $p$ is a prime then $\BF_p, \BZ_p,\BQ_p$ stand for the $p$-element (finite) field, the ring of $p$-adic integers, the field of $p$-adic numbers respectively.
If $E$ is a number field of degree $n=[E:\BQ]$ then $r_E$ and $s_E$ are nonnegative integers such that the $\BR$-algebra $E_{\BR}=E\otimes_{\BQ}\BR$
is isomorphic to a product $\BR^{r_E} \times \BC^{s_E}$. (In other words, $r_E$ is the number of ``real'' field embeddings $E \hookrightarrow \BR$ and $2s_E$
is the number of ``imaginary'' field embeddings $E \hookrightarrow \BC$, whose images do {\sl not} lie in $\BR$.) In particular,
\begin{equation}
\label{rsE}
[E:\BQ]=r_E+2s_E.
\end{equation}

Let $X$ be a connected compact complex K\"ahler manifold, 
  $\mathrm{H}^2(X,\BQ)$ its second rational cohomology group equipped with the canonical rational Hodge structure,  i.e., there is the Hodge decomposition
$$\mathrm{H}^2(X,\BQ)\otimes_{\BQ}\BC=\mathrm{H}^2(X,\BC)=\mathrm{H}^{2,0}(X)\oplus \mathrm{H}^{1,1}(X)\oplus \mathrm{H}^{2,0}(X)$$
where 
$\mathrm{H}^{2,0}(X)=\Omega^2(X)$ is the space of holomorphic $2$-forms on $X$,
$\mathrm{H}^{0,2}(X)$ is the ``complex-conjugate'' of $\mathrm{H}^{2,0}(X)$ and 
$\mathrm{H}^{1,1}(X)$ coincides with its own ``complex-conjugate'' (see 
\cite[Sections 2.1--2.2]{DeligneH2}, \cite[Ch. VI-VII]{V})).
 The following property of $X$ was introduced and studied by 
F. Campana \cite[Definition 3.3]{Campana}.    (Recently, it was used in the study of coisotropic and lagrangian submanifolds of symplectic manifolds \cite{AC}.)

\begin{defn}
\label{CampDef}
A  manifold $X$ is {\sl irreducible in weight 2} (irr\'eductible en poids 2) if it
enjoys the following property.

Let $H$ be a rational Hodge substructure of $\mathrm{H}^2(X,\BQ)$ such that
$$H_{\BC}\cap \mathrm{H}^{2,0}(X)\ne \{0\}$$
where $H_{\BC}:=H\otimes_{\BQ}\BC$.

Then $H_{\BC}$ contains the whole $\mathrm{H}^{2,0}(X)$.
\end{defn}

Our aim  is to study  complex tori $T$ that enjoy this property. Namely, we discuss their endomorphism algebras, automorphism groups and Hodge groups.

Let $T=V/\Lambda$ be a complex torus of positive dimension $g$ where $V$  is a $g$-dimensional complex vector space, and $\Lambda$ is a discrete lattice of rank $2g$ in $V$.
 One may naturally identify $\Lambda$ with the first integral homology group $\mathrm{H}_1(T,\BZ)$ of $T$ and
$$\Lambda_{\BQ}=\Lambda\otimes \BQ=\{v \in V \mid \exists n\in \BZ\setminus \{0\} \ \text{ such that } nv \in \Lambda\}$$
with the first rational homology group $\mathrm{H}_1(T,\BQ)$ of $T$. 
 There are also natural isomorphisms of real vector spaces
 $$\Lambda \otimes \BR =\Lambda_{\BQ}\otimes_{\BQ}\BR \to V, \ \lambda\otimes r \mapsto r\lambda$$
 that may be viewed as isomorphisms related to the first real cohomology group $\mathrm{H}_1(T,\BR)$ of $T$:
 $$\mathrm{H}_1(T,\BR)=\mathrm{H}_1(T,\BZ)\otimes\BR =\mathrm{H}_1(T,\BQ)\otimes_{\BQ}\BR \to V.$$
 In particular, there is a canonical isomorphism of real vector spaces
 \begin{equation}
 \label{VR}
  \mathrm{H}_1(T,\BR)=V,
  \end{equation}
  and a canonical isomorphism of complex vector spaces
  \begin{equation}
 \label{VC}
  \mathrm{H}_1(T,\BC) =\mathrm{H}_1(T,\BQ)\otimes_{\BQ}\BC=\mathrm{H}_1(T,\BR)\otimes_{\BR}\BC=V\otimes_{\BR}\BC=:V_{\BC}
  \end{equation}
  where $ \mathrm{H}_1(T,\BC)$ is the first complex homology group of $T$.
  
There are natural isomorphisms of $\BR$-algebras
$$\End_{\BZ}(\Lambda)\otimes \BR \cong \End_{\BR}(V), \  u\otimes r\mapsto ru,$$
$$\End_{\BQ}(\Lambda_{\BQ})\otimes \BR \cong \End_{\BR}(V), \  u\otimes r\mapsto ru,$$
which give rise to the natural ring embeddings
\begin{equation}
\label{embedZQR}
\End_{\BZ}(\Lambda)\subset \End_{\BQ}(\Lambda_{\BQ})\subset \End_{\BR}(V)\subset  \End_{\BR}(V)\otimes_{\BR}\BC=\End_{\BC}(V_{\BC}).
\end{equation}
Here the structure of a $2g$-dimensional {\sl complex} vector space on $V_{\BC}$ is defined by
 $$z (v\otimes s)=v\otimes zs \quad \forall v\otimes s \in V\otimes_{\BR}\BC=V_{\BC},  \ z \in \BC.$$
 If $u \in \End_{\BR}(V)$ then we write $u_{\BC}$ for the corresponding $\BC$-linear operator in $V_{\BC}$, i.e.,
 \begin{equation}
 \label{notationAbuse}
 u_{\BC}(v\otimes z)=u(v)\otimes z \quad \forall u \in V, z \in \BC, v\otimes z\in V_{\BC}.
 \end{equation}
 \begin{rem}
 \label{notAB}
 Sometimes, we will  identify $\End_{\BR}(V)$ with its image 
 $\End_{\BR}(V)\otimes 1\subset \End_{\BC}(V_{\BC})$ and write $u$ instead of $u_{\BC}$, slightly abusing notation.
 \end{rem}
 
 As usual, one may naturally extend the complex conjugation $z \mapsto \bar{z}$ on $\BC$ to the $\BC$-antilinear involution
 $$V_{\BC} \to V_{\BC},  \ w \mapsto \bar{w}, \
 v\otimes z \mapsto \overline{v\otimes z}=v \otimes \bar{z},$$
 which is usually called the complex conjugation on $V_{\BC}$. Clearly, 
 \begin{equation}
 \label{uBar}
 u_{\BC}(\bar{w})=\overline{u(w)} \quad \forall u \in  \End_{\BR}(V), w \in V_{\BC}.
 \end{equation}
 This implies easily that the set of fixed points of the involution is
 $$V=V\otimes 1 \subset V_{\BC}.$$

Let $\End(T)$ be the endomorphism ring of the complex commutative Lie group $T$ and
$\End^0(T)=\End(T)\otimes\BQ$ the corresponding endomorphism algebra, which is a finite-dimensional algebra over 
$\BQ$,
see \cite{OZ,BL,BZ}.
There are well known canonical isomorphisms
$$\End(T)=\End_{\BZ}(\Lambda)\cap  \End_{\BC}(V), \quad
\End^{0}(T)=\End_{\BQ}(\Lambda_{\BQ})\cap  \End_{\BC}(V).$$

Let  $g \ge 2$ and $$\mathrm{H}^2(T,\BQ)=\wedge^2_{\BQ}(\Lambda_{\BQ},\BQ)$$ be the {\sl second rational cohomology group} of $T$,
which carries a natural  rational Hodge structure of weight two:
 $$\mathrm{H}^2(T,\BC)= \mathrm{H}^2(T,\BQ)\otimes_{\BQ}\BC=\mathrm{H}^{2,0}(T)\oplus \mathrm{H}^{1,1}(T)\oplus \mathrm{H}^{0,2}(T)$$
 where $\mathrm{H}^{2,0}(T)=\Omega^2(T)$ is the $g(g-1)/2$-dimensional space of holomorphic $2$-forms on $T$.

\begin{defn}
\label{simple2}
Let $g=\dim(T)\ge 2$. We say that $T$ is {\sl $2$-simple} if it is {\sl irreducible in weight 2}, i.e.,
enjoys the following property.

Let $H$ be a rational Hodge substructure of $\mathrm{H}^2(T,\BQ)$ such that
$$H_{\BC}\cap \mathrm{H}^{2,0}(T)\ne \{0\}$$
where $H_{\BC}:=H\otimes_{\BQ}\BC$.

Then $H_{\BC}$ contains the whole $\mathrm{H}^{2,0}(T)$.
\end{defn}

\begin{rem}
We call such  complex tori $2$-simple, because they are simple in the usual meaning of this word if $g>2$, see Theorem \ref{main}(i) below.

\end{rem}

\begin{ex} (See \cite[Example 3.4(2)]{Campana}.)
If $g=2$ then $\dim_{\BC}(\mathrm{H}^{2,0}(T))=1$. This implies that (in the notation of Definition 
\ref{simple2}) if $H_{\BC}\cap \mathrm{H}^{2,0}(T)\ne \{0\}$ then $H_{\BC}$ contains the whole
$\mathrm{H}^{2,0}(T)$. Hence, every $2$-dimensional complex torus is $2$-simple.
\end{ex}

In what follows we write $\Aut(T)=\End(T)^{*}$ for the automorphism group of the complex Lie group $T$.
We will need the following well known definition.

\begin{defn} A number field $E$ is called {\sl primitive} if either $E=\BQ$ or
the only proper subfield of $E$ is $\BQ$.
\end{defn}

Our main result is the following assertion.

\begin{thm}
\label{main}
Let $T$ be a complex torus of dimension $g\ge 3$.  Suppose that $T$ is $2$-simple.

Then $T$  enjoys the following properties.

\begin{itemize}
\item[(i)]
$T$ is simple.
\item[(ii)]
If $E$ is any subfield of $\End^0(T)$ then it is a number field, whose degree over $\BQ$
is either $1$  or $g$ or $2g$.
\item[(iii)]

 $\End^0(T)$ is a number field, whose degree over $\BQ$
is either $1$  or $g$ or $2g$.

\item[(iv)]
If   $[\End^0(T):\BQ]=1$ then
$$\End^0(T)=\BQ,  \End(T)=\BZ, \ \Aut(T)=\{\pm 1\}.$$

\item[(v)] If $E=\End^0(T)$ and
 $[E:\BQ]=2g$ then $E$ is a purely imaginary number field and   $\Aut(T)\cong\{\pm 1\}\times \BZ^{g-1}$.
In addition,  if $E$ is not primitive then it contains precisely one proper subfield except $\BQ$, and this subfield has degree $g$ over $\BQ$.

\item[(vi)] If $E=\End^0(T)$ and $[E:\BQ]=g$ then  $E$ is a primitive number field and
$\Aut(T) \cong \BZ^{d} \times \{\pm 1\}$ where the positive  integer $d$ equals $r_E+s_E-1$.
In particular,  
$$\frac{1}{2} \le \frac{g}{2}-1 \le d \le g-1.$$
In addition, if $T$ is a complex abelian variety then $E$ is a primitive totally real number field and $d=g-1$.
\end{itemize}
\end{thm}

\begin{rem} 
\label{translation}
\begin{itemize}
\item[(i)]
It is well known (and can be easily checked) that $T$ is {\sl simple} if and only if the rational Hodge structure on $\Lambda_{\BQ}=\mathrm{H}_1(T,\BQ)$  is {\sl irreducible}.\footnote{A rational Hodge structure $H$ is called {\sl irreducible} or {\sl simple} if its only rational Hodge substructures are $H$ itself and $\{0\}$ \cite[Sect. 2.2]{Charles}.}
\item[(ii)]

We may view $\mathrm{H}^2(T,\BQ)$ as the $\BQ$-vector subspace $\HH^2(T,\BQ)\otimes 1$
 of $\HH^2(T,\BQ)\otimes_{\BQ}\BC=\HH^2(T,\BC)$.
Let us consider the $\BQ$-vector (sub)space 
$$\HH^{1,1}(T,\BQ):=\HH^2(T,\BQ)\cap \HH^{1,1}(T)$$
of $2$-dimensional {\sl Hodge cycles} on $T$.    Notice that the {\sl irreducibility}
of the rational Hodge structure on $\Lambda_{\BQ}$ implies the {\sl complete reducibility}
\footnote{A rational Hodge structure is called completely reducible it it splits into a direct sum of simple rational Hodge structures.}
 of the rational Hodge structure
on $\HH^2(T,\BQ)=\Hom_{\BQ}\left(\wedge^2_{\BQ}\Lambda_{\BQ},\BQ\right)$.  (It follows from the reductiveness of the Mumford-Tate group of a simple torus \cite[Sect. 2.2]{Charles}.)
In light of (i) and Theorem \ref{main}(i), a complex torus $T$ of dimension $>2$
is $2$-simple if and only if it is simple and $\HH^2(T,\BQ)$ splits into a direct sum of $\HH^{1,1}(T,\BQ)$ and an {\sl irreducible} rational Hodge substructure.
\end{itemize}
\end{rem}

\begin{thm}
\label{existRealT}
Let $g \ge 3$ be an integer. Let $\mathbf{r},\mathbf{s}$ be nonnegative integers such that
$$\mathbf{r}+2\mathbf{s}=g.$$

Then there exists a $2$-simple torus $T$ of degree $g$ that enjoys the following properties.

The endomorphism algebra $\End^0(T)$ is a number field $E$ such that
$$[E:\BQ]=g, \ r_E= \mathbf{r}, \  s_E= \mathbf{s}.$$

In particular, if $d$ is an integer such that
$$\frac{g}{2}-1 \le d \le g-1$$
then there exists a $g$-dimensional $2$-simple complex torus $T$ such that
$$\Aut(T) \cong \BZ^{d} \times \{\pm 1\}.$$ 
\end{thm}

The paper is organized as follows.
We prove Theorem \ref{main} in Section \ref{proofMain}, using  explicit constructions related to the Hodge structure on $\Lambda_{\BQ}$
that will be discussed in Section \ref{hodgeS}.  Section \ref{almostTran} deals with (mostly well known) results about number fields that will be used
in the computations of Hodge groups of complex tori.
In Section \ref{HDG} we discuss general properties of Hodge groups of $2$-simple tori.
In Section \ref{DegreeG} we concentrate on the case when the endomorphism algebra is a number field of degree $g$.
Section \ref{semiLin} contains more or less standard results from ``Semilinear Algebra'' that we use in
in Section \ref{RepsKk},  in order to study representations of semisimple Lie algebras over arbitrary fields of characteristic zero. 
Using results of Section \ref{RepsKk}, we study in Section \ref{noEnd}  Hodge groups of $2$-simple tori without nontrivial endomorphisms.

This paper may be viewed as a follow up of \cite{OZ} and \cite{BZ}.

{\bf Acknowledgements} I am grateful to Fr\'ed\'eric Campana and Ekaterina Amerik for interesting stimulating questions.  My special thanks go to
Grigori Olshanski for a very informative letter about {\sl plethysm}. I am grateful to Tatiana Bandman and the referee, whose thoughtful comments helped to improve the exposition.

\section{Hodge structures}
\label{hodgeS}

  \begin{sect}
  It is well known that $\Lambda_{\BQ}=\mathrm{H}_1(T,\BQ)$
  carries the natural structure of a rational Hodge structure of weight $-1$. Let us recall the construction.
  Let $J: V \to V$ be the multiplication by $\ib=\sqrt{-1}$, which is viewed as an element of  $\End_{\BR}(V)$
  such that 
  $$J^2=-1.$$
  Hence, $J_{\BC}^2=-1$ in $\End_{\BC}(V_{\BC})$ and we define
   two mutually complex-conjugate  $\BC$-vector subspaces (of the same dimension)
 $\mathrm{H}_{-1,0}(T)$ and $\mathrm{H}_{0,-1}(T)$ of $V_{\BC}$
 as the eigenspaces $V_{\BC}(\mathbf{i})$ and $V_{\BC}(-\mathbf{i})$ of $J_{\BC}$ attached to eigenvalues $\ib$ and $-\ib$ respectively. Clearly,
 $$V_{\BC}=V_{\BC}(\mathbf{i})\oplus V_{\BC}(-\mathbf{i})=\mathrm{H}_{-1,0}(T)\oplus \mathrm{H}_{0,-1}(T),$$
 which defines the rational Hodge structure of weight $-1$ on $\Lambda_{\BQ}$, in light of
 $V_{\BC}=\Lambda_{\BQ}\otimes_{\BQ}\BC$. It also follows that both $\mathrm{H}_{-1,0}(T)$ and $\mathrm{H}_{0,-1}(T)$
 have the same  complex dimension $2g/2=g$.

 Recall that $V$ is a complex vector space. I claim that the map
 \begin{equation}
 \label{VV}
 \Psi: V \to V_{\BC}(\mathbf{i})=\mathrm{H}_{-1,0}(T), \ v \mapsto Jv\otimes 1+v\otimes \mathbf{i}
 \end{equation}
 is an  isomorphism of complex vector spaces.  Indeed, first, $\Psi$ defines a homomorphism of real vector spaces
 $V \to V_{\BC}$. Second, if $v\in V$ then
 $$J_{\BC}( Jv\otimes 1+v\otimes \mathbf{i})=J^2v\otimes 1+Jv\otimes \mathbf{i}=-v\otimes 1+Jv\otimes \mathbf{i}=
 \ib(Jv\otimes 1+v\otimes \ib),$$
 i.e., $Jv\otimes 1+v\otimes \ib\in V_{\BC}(\mathbf{i})=\mathrm{H}_{-1,0}(T)$ and therefore the map \eqref{VV} 
 is defined correctly. Third, taking into account that 
 $J$ is an automorphism of $V$ and $V_{\BC}=V\otimes 1\oplus V\otimes\ib$, we conclude that $\Psi$ is an injective homomorphism
 of real vector spaces and a dimension argument implies that it is actually an isomorphism. It remains to check that $\Psi$ is $\BC$-linear,
 i.e.,  $$\Psi(Jv)=\ib \Psi(v).$$  Let us do it. We have
 $$\Psi(Jv)=J(Jv)\otimes 1+Jv\otimes\ib=-v\otimes 1+Jv\otimes \ib=\ib (Jv\otimes 1+v\otimes \ib)=\ib \Psi(v).$$
 Hence, $\Psi$ is a $\BC$-linear isomorphism and we are done.
 
 Now suppose that  $u \in \End_{\BR}(V)$ commutes with $J$, i.e., $u\in \End_{\BC}(V)$. Then
 \begin{equation}
 \label{uPsiBC}
 \Psi\circ u=u_{\BC}\circ\Psi.
  \end{equation}
  In particular, $\mathrm{H}_{-1,0}(T)$ is $u_{\BC}$-invariant. Indeed, if $v\in V$ then
  $$ \Psi\circ u(v)=J u(v)\otimes 1+u(v)\otimes\ib=uJ(v)\otimes 1+u_{\BC}(v\otimes \ib)=u_{\BC}(J(v)\otimes 1)+u_{\BC}(v\otimes \ib)=u_{\BC}\circ\Psi(v),$$
  which proves our claim.
 
 Similarly, there is an anti-linear isomorphism of complex vector spaces
 $$V \to V_{\BC}(\mathbf{-i})=\mathrm{H}_{0,-1}(T), \ v \mapsto Jv\otimes 1-v\otimes \mathbf{i}.$$
 \end{sect}

 It is also well known that there is a canonical isomorphism of rational Hodge structures of weight 2
$$\mathrm{H}^2(T,\BQ)=\Hom_{\BQ}\large(\wedge^2_{\BQ}\mathrm{H}_1(T,\BQ),\BQ\large)$$
where the Hodge components $\mathrm{H}^{p,q}(T)$ ($p,q \ge 0, p+q=2$) are  as follows.
\begin{equation}
\label{Hodge2T}
\mathrm{H}^{2,0}(T)=\Hom_{\BC}\large(\wedge^2_{\BC}\mathrm{H}_{-1,0}(T),\BC\large), \quad
\mathrm{H}^{0,2}(T)=\Hom_{\BC}(\wedge^2_{\BC}\mathrm{H}_{0,-1}(T),\BC),
\end{equation}
$$\mathrm{H}^{1,1}(T)=\Hom_{\BC}(\mathrm{H}_{-1,0}(T),\BC)\wedge \Hom_{\BC}(\mathrm{H}_{0,-1}(T),\BC)\cong$$
$$\Hom_{\BC}(\mathrm{H}_{-1,0}(T),\BC)\otimes_{\BC} \Hom_{\BC}(\mathrm{H}_{0,-1}(T),\BC).$$
Clearly,
$$\dim_{\BC}(\mathrm{H}^{2,0}(T))=\frac{g(g-1)}{2}.$$

\section{Endomorphism Fields and Automorphism Groups}
\label{proofMain}

\begin{proof}[Proof of Theorem \ref{main}]
Let $T$ be a $2$-simple complex torus and assume that
$$g=\dim(T)\ge 3.$$

(i) Suppose that $T$ is {\sl not} simple. This means that there is a proper complex subtorus
$S=W/\Gamma$ where $W$ is a complex vector subspace of $V$ with
$$0<d=\dim_{\BC}(W)<\dim_{\BC}(V)=g$$
such that
$$\Gamma=W\cap \Lambda$$
is a discrete lattice of rank $2d$ in $W$.  Then the quotient $T/S$ is a complex torus of positive dimension $g-d$.

Let $H \subset \mathrm{H}^2(T,\BQ)$ be the image of the canonical
{\sl injective}  homomorphism of rational Hodge structures
$\mathrm{H}^2(T/S,\BQ) \hookrightarrow \mathrm{H}^2(T,\BQ)$ induced by the quotient map $T \to T/S$ of complex tori. Clearly, $H$ is a rational Hodge
substructure of  $\mathrm{H}^2(T,\BQ)$ and its $(2,0)$-component
$$H^{2,0}\subset H_{\BC}$$
has $\BC$-dimension
$$\dim_{\BC}(H^{2,0})=\dim_{\BC}(\mathrm{H}^{2,0}(T/S))=\frac{(g-d)(g-d-1)}{2}<\frac{g(g-1)}{2}=
\dim_{\BC}(\mathrm{H}^{2,0}(T)).$$
In light of the $2$-simplicity of $T$, 
$$\dim_{\BC}(\HH^{2,0})=0,$$
which implies that 
$$g-d=1.$$
On the other hand, let $\tilde{H}$ be the kernel of the canonical {\sl surjective}
homomorphism of rational Hodge structures
$\mathrm{H}^2(T,\BQ) \twoheadrightarrow \mathrm{H}^2(S,\BQ)$ induced by the inclusion map $S \subset T$ of complex tori. Clearly, $\tilde{H}$ is a rational Hodge
substructure of $\mathrm{H}^2(T,\BQ)$.  Notice that the induced homomorphism of  $(2,0)$-components
$\mathrm{H}^{2,0}(T) \to \mathrm{H}^{2,0}(S)$ is also surjective, because every holomorphic $2$-form on $S$ obviously extends to a holomorphic $2$-form on $T$.
This implies that the
 $(2,0)$-component
$$\tilde{H}^{2,0}\subset \tilde{H}_{\BC}$$
of $\tilde{H}$
has $\BC$-dimension
$$\dim_{\BC}(\tilde{H}^{2,0})=\dim_{\BC}(\mathrm{H}^{2,0}(T))-\dim_{\BC}(\mathrm{H}^{2,0}(S))=
\frac{g(g-1)}{2}-\frac{d(d-1)}{2}>0.$$
In light of the $2$-simplicity of $T$, 
$$\dim_{\BC}(\tilde{H}^{2,0})=\dim_{\BC}(\mathrm{H}^{2,0}(T))=\frac{g(g-1)}{2},$$
which implies that $\frac{d(d-1)}{2}=0$, i.e., $d=1$. Taking into account that $g-d=1$, we get $g=1+1=2$, which is not true.
The obtained contradiction proves that $T$ is simple and (i) is proven.  In particular, $\End^0(T)$ is a division algebra over $\BQ$.

In order to handle (ii), let us assume that $E$ is a subfield of $\End^0(T)$.  The simplicity of $T$ implies that $1\in E$
is the identity automorphism of $T$.
Then $\Lambda_{\BQ}$ becomes a faithful $E$-module.
This implies that $E$ is a number field and  $\Lambda_{\BQ}$ is an $E$-vector space of finite positive dimension
$$d_E=\frac{2g}{[E:\BQ]}.$$  This implies that $V_{\BC}=\Lambda_{\BQ}\otimes_{\BQ}\BC$ is a free $E\otimes_{\BQ}\BC$-module of rank $d_E$.
Clearly, both $\mathrm{H}_{-1,0}(T)$ and $\mathrm{H}_{0,-1}(T)$ are $E\otimes_{\BQ}\BC$-submodules of its direct sum $V_{\BC}$.
Let
$$\mathrm{tr}_{E/\BQ}: E \to \BQ$$
be the trace map attached to the  field extension $E/\BQ$ of finite degree. Let
$$\Hom_E\big(\wedge_E^2 \Lambda_{\BQ},E\big)$$
be the $\frac{d_E(d_E-1)}{2}$-dimensional $E$-vector space of alternating $E$-bilinear forms on $\Lambda_{\BQ}$; it carries the natural structure
of a rational Hodge structure of $\BQ$-dimension $[E:\BQ] \cdot \frac{d_E(d_E-1)}{2}$.  There is the natural embedding of rational Hodge structures
\begin{equation}
\label{HE2}
\Hom_E\left(\wedge_E^2 \Lambda_{\BQ},E\right) \hookrightarrow \Hom_{\BQ}\left(\wedge_{\BQ}^2 \Lambda_{\BQ},\BQ\right)=\mathrm{H}^2(T,\BQ), \
\phi_E \mapsto \phi:=\mathrm{tr}_{E/\BQ}\circ \phi_E,
\end{equation}
$$\ \forall \phi_E \in \Hom_E\left(\wedge_E^2 \Lambda_{\BQ},E\right),$$
i.e.,
\begin{equation}
\label{HE3}
\phi(\lambda_1,\lambda_2)=\mathrm{tr}_{E/\BQ} \big(\phi_E(\lambda_1,\lambda_2)\big) \ \forall \lambda_1,\lambda_2\in \Lambda_{\BQ}.
\end{equation}
The image of 
$\Hom_E\left(\wedge_E^2 \Lambda_{\BQ},E\right)$  in
$\Hom_{\BQ}\left(\wedge_{\BQ}^2 \Lambda_{\BQ},\BQ\right)=\mathrm{H}^2(T,\BQ)$ 
coincides with the $\BQ$-vector subspace 
\begin{equation}
\label{psiE}
H_E:=\{\phi \in \Hom_{\BQ}\left(\wedge_{\BQ}^2 \Lambda_{\BQ},\BQ\right) \mid
\phi(u \lambda_1,\lambda_2)=\phi(\lambda_1,u\lambda_2) \ \forall  u \in E,  \ \lambda_1,\lambda_2 \in \Lambda_{\BQ}\}.
\end{equation}
Indeed, it is obvious that the image lies in $H_E$. In order to check that the image coincides with the whole subspace $H_E$, let us construct
the inverse map
$$H_E \to \Hom_E\left(\wedge_E^2 \Lambda_{\BQ},E\right),  \ \phi \mapsto \phi_E$$ 
to   \eqref{HE2}
as follows.  If $\lambda_1,\lambda_2\in \Lambda_{\BQ}$
then there is a  $\BQ$-linear map
\begin{equation}
\label{phiE}
\Phi: E \mapsto \BQ,  \ u \mapsto \phi(u\lambda_1,\lambda_2)=\phi(\lambda_1,u\lambda_2)=-\phi(u\lambda_2,\lambda_1)=-\phi(\lambda_2,u\lambda_1).
\end{equation}
The properties of the trace map imply that there exists precisely one $\beta \in E$ such that
$$\Phi(u)=\tr_{E/\BQ}(u\beta) \ \forall u \in E.$$
Let us put
$$\\phi_E(\lambda_1,\lambda_2):=beta.$$
It follows from \eqref{phiE} that $\phi_E \in \Hom_E\left(\wedge_E^2 \Lambda_{\BQ},E\right)$.
In addition,
$$\mathrm{tr}_{E/\BQ}(\phi_E(\lambda_1,\lambda_2))=\mathrm{tr}_{E/\BQ}(\beta)=\mathrm{tr}_{E/\BQ}(1\cdot \beta)=\Phi(1)=
 \phi(\lambda_1,\lambda_2),$$
 which proves that $\phi\mapsto \phi_E$ is indeed the inverse map, in light of \eqref{HE3}.
 
 Clearly, $H_E$ is a rational Hodge substructure of $\mathrm{H}^2(T,\BQ)$.

By $2$-simplicity of $T$, the $\BC$-dimension of the $(2,0)$-component $H_E^{(2,0)}$ of $H_E$  is either $0$ or $g(g-1)/2$.
Let us express this dimension explicitly in terms of $g$ and $[E:\BQ]$.

In order to do that, let us consider the   set $\Sigma_E$ of  all field embeddings $\sigma: E \hookrightarrow \BC$, which consists of $[E:\BQ]$-elements. We have
\begin{equation}
\label{SigmaE}
E_{\BC}:=E\otimes_{\BQ}\BC=\bigoplus_{\sigma\in\Sigma_E}\BC_{\sigma} \ \text{ where } \BC_{\sigma}=E\otimes_{E,\sigma}\BC=\BC,
\end{equation}
which gives us a splitting of $E_{\BC}$-modules
\begin{equation}
\label{SigmaV}
V_{\BC}=\oplus_{\sigma \in \Sigma_E}V_{\sigma}=\bigoplus_{\sigma \in \Sigma_E}\left(\mathrm{H}_{-1,0}(T)_{\sigma}\oplus \mathrm{H}_{0,-1}(T)_{\sigma}\right)
\end{equation}
where for all $\sigma \in \Sigma_E$ we define
$$\mathrm{H}_{-1,0}(T)_{\sigma}: =\BC_{\sigma}\mathrm{H}_{-1,0}(T)=
\{x \in \mathrm{H}_{-1,0}(T) \mid u_{\BC}x=\sigma(u)x \ \forall u \in E\}\subset \mathrm{H}_{-1,0}(T);$$
$$n_{\sigma}:=\dim_{\BC}(\mathrm{H}_{-1,0}(T)_{\sigma});$$

$$\mathrm{H}_{0,-1}(T)_{\sigma}: =\BC_{\sigma}\mathrm{H}_{0,-1}(T)= \{x \in \mathrm{H}_{0,-1}(T) \mid u_{\BC}x=\sigma(u)x \ \forall u \in E\}\subset \mathrm{H}_{0,-1}(T);$$
$$m_{\sigma}:=\dim_{\BC}(\mathrm{H}_{0,-1}(T)_{\sigma});$$

$$V_{\sigma}=\BC_{\sigma}=\BC_{\sigma}V_{\BC}=\{x \in  V_{\BC} \mid u_{\BC}x=\sigma(u)x \ \forall u \in E\}=\mathrm{H}_{-1,0}(T)_{\sigma}\oplus \mathrm{H}_{0,-1}(T)_{\sigma} $$

Since $\mathrm{H}_{-1,0}(T)\oplus \mathrm{H}_{0,-1}(T)=V_{\BC}$ is a free $E_{\BC}$-module of rank $d_E$,
its direct summand $V_{\sigma}$ is a vector space of dimension $d_E$ 
over $\BC_{\sigma}=\BC$
and therefore
\begin{equation}
\label{nplusm}
n_{\sigma}+m_{\sigma}=d_E
\end{equation}
 for all $\sigma$.
Since   $\mathrm{H}_{-1,0}(T)$ and  $\mathrm{H}_{0,-1}(T)$ are mutually complex-conjugate subspaces of $V_{\BC}$,    it follows from \eqref{uBar}
that $$m_{\sigma}=n_{\bar{\sigma}} \ \ \text{ where }
\bar{\sigma}:E \hookrightarrow \BC,   \ u \mapsto \overline{\sigma(u)}$$ is the {\sl complex-conjugate}  of $\sigma$.  Therefore, in light of \eqref{nplusm},
\begin{equation}
\label{sigmaBaRsigma}
n_{\sigma}+n_{\bar{\sigma}}=d_E \ \forall \sigma.
\end{equation}
 We have
 \begin{equation}
 \label{sumS}
\sum_{\sigma\in \Sigma_E}n_{\sigma}=\sum_{\sigma\in\Sigma_E}\dim_{\BC}(\mathrm{H}_{-1,0}(T)_{\sigma})=   \dim_{\BC}(\mathrm{H}_{-1,0}(T))=g.
\end{equation}
Let us consider the complexification of $H_E$
$$H_{E,\BC}:=H_E\otimes_{\BQ}\BC\subset \Hom_{\BQ}\left(\wedge^2\Lambda_{\BQ},\BQ\right)\otimes_{\BQ}\BC=$$
$$\Hom_{\BC}\left(\wedge_{\BC}^2(\Lambda_{\BQ}\otimes_{\BQ}\BC),\BC\right)=
\Hom_{\BC}\left(\wedge^2V_{\BC},\BC\right).$$
In light of \eqref{psiE},
\begin{equation}
\label{EpsiC}
H_{E,\BC}=\{\phi \in \Hom_{\BC}\left(\wedge_{\BC}^2 V_{\BC},\BC\right)\mid \phi(u_{\BC}x,y)=\phi(x,u_{\BC}y) \ \forall u\in E; \ x,y\in V_{\BC}\}
\end{equation}
$$=\{\phi \in \Hom_{\BC}\left(\wedge_{\BC}^2V_{\BC},\BC\right)\mid \phi(u_{\BC}x,y)=\phi(x,u_{\BC}y) \ \forall u\in E_{\BC}; \ x,y\in V_{\BC}\}.$$
In particular, if $\sigma, \tau \in \Sigma_E$ are {\sl distinct} field embeddings then for all $\phi \in H_{E,\BC}$
$$\phi(V_{\sigma},V_{\tau})=\phi(V_{\tau},V_{\sigma})=\{0\}.$$

This implies that
\begin{equation}
\label{secondH}
H_{E,\BC}=
\oplus_{\sigma\in \Sigma_E}\Hom_{\BC}\left(\wedge^2_{\BC}V_{\sigma},\BC\right)
\end{equation}
$$=
\bigoplus_{\sigma\in \Sigma_E}\Hom_{\BC}\big(\wedge_{\BC}^2 \left(\mathrm{H}_{-1,0}(T)_{\sigma} \bigoplus  \mathrm{H}_{0,-1}(T)_{\sigma}\right) ,\BC\big).$$
In light of \eqref{Hodge2T}, the $(2,0)$-Hodge component of $H_{E,\BC}$ is
\begin{equation}
\label{secondF}
H_E^{(2,0)}=\oplus_{\sigma\in \Sigma_E}\Hom_{\BC}\left(\wedge_{\BC}^2 \mathrm{H}_{-1,0}(T)_{\sigma} ,\BC\right) \ \text{ and } \
\dim_{\BC}(H_E^{(2,0)})=\sum_{\sigma\in\Sigma_E} \frac{n_{\sigma}(n_{\sigma}-1)}{2}.
\end{equation}

This implies that
$\dim_{\BC}(H_E^{(2,0)})=0$ if and only if all $n_{\sigma}$ are in $\{0,1\}$. If this is the case then, in light of \eqref{sigmaBaRsigma},
$d_E \in\{1,2\}$, i.e., $[E:\BQ]=2g$ or $g$.

On the other hand, it follows from \eqref{sumS} combined with the second formula in \eqref{secondF} that
$\dim_{\BC}(H_E^{(2,0)})=g(g-1)/2$ if and only if there is precisely one $\sigma$ with $n_{\sigma}=g$
(and all  the other multiplicities $n_{\tau}$ are $0$). This implies that either $d_E=2g$ and $E=\BQ$, 
or $d_E=g$ and $E$ is an imaginary quadratic field with the pair of the field embeddings
$$\sigma,\bar{\sigma}: E \hookrightarrow \BC$$
such that
$$n_{\sigma}=g, \ n_{\bar{\sigma}}=0.$$

It is therefore enough to rule out the case  $d_E=g$.  By way of contradiction, 
 assume that  $d_E=g$. Then  $E$ is an imaginary quadratic field; in addition, if
$$u \in E\subset \End_{\BQ}(\Lambda_{\BQ})\subset \End_{\BR}(V)$$ then $u_{\BC}$ acts on $\mathrm{H}_{-1,0}(T)$ as multiplication by $\sigma(u) \in \BC$.
In light of \eqref{uBar},  $u_{\BC}$ acts on the complex-conjugate  subspace $\mathrm{H}_{0,-1}(T)$ as multiplication by $\overline{\sigma(u)}=\bar{\sigma}(u) \in \BC$.
Since $E$ is an imaginary quadratic field,  there are a positive integer $D$ and $\alpha\in E$ such that
$\alpha^2=-D$ and $E=\BQ(\alpha).$ It follows that
$\sigma(\alpha)=\pm \ib \sqrt{D}$. Replacing if necessary $\alpha$ by $-\alpha$, we may and will assume that
$$\sigma(\alpha)=\ib \sqrt{D}$$
and therefore $\alpha_{\BC}$ acts on $\mathrm{H}_{-1,0}(T)$ as multiplication by $\ib \sqrt{D}$. Hence, 
$\alpha_{\BC}$ acts on $\mathrm{H}_{0,-1}(T)$ as multiplication by $\overline{\ib \sqrt{D}}=-\ib \sqrt{D}$. Since
$$V_{\BC}=\mathrm{H}_{-1,0}(T)\oplus\mathrm{H}_{0,-1}(T),$$
we get 
$\alpha_{\BC}= \sqrt{D}J_C$ and therefore 
$$\alpha= \sqrt{D}J.$$
 This implies that the 
centralizer $\End^0(T)$ of $J$ in $\End_{\BQ}(\Lambda_{\BQ})$
coincides with the centralizer of $\alpha$ in $\End_{\BQ}(\Lambda_{\BQ})$, which, in turn, coincides with the centralizer $\End_E(\Lambda_{\BQ})$ of $E$ in $\End_{\BQ}(\Lambda_{\BQ})$, i.e.,
 $$\End^0(T)=\End_E(\Lambda_{\BQ}) \cong \Mat_{d_E}(E).$$
  This is a matrix algebra, which is not a division algebra, because $d_E=g>1$. This contradicts  the simplicity of $T$.
The obtained contradiction rules out the case $d_E=g$.
This ends the proof of (ii).

In order to prove (iii), recall that $\End^0(T)$ is a division algebra over $\BQ$, thanks to the simplicity of $T$ \cite{OZ}.
Hence $\Lambda_{\BQ}$ is a free  $\End^0(T)$-module of finite positive rank
and therefore 
\begin{equation}
\label{divisionP}
\dim_{\BQ}(\End^0(T))\mid 2g,
\end{equation}
because
$2g=\dim_{\BQ}(\Lambda_{\BQ})$. 
We will apply several times  the already proven assertion (ii) to various subfields of  $\End^0(T)$.

Suppose that  $\End^0(T)$ is {\sl not} a field and let $\mathcal{Z}$ be its center. Then $\mathcal{Z}$ is a number field and 
there is an integer $d>1$ such that
$\dim_{\mathcal{Z}}(\End^0(T))=d^2$ and therefore   
$$\dim_{\BQ}(\End^0(T))=d^2 \cdot [\mathcal{Z}:\BQ]$$
divides $2g$, thanks to \eqref{divisionP}.   Since $\mathcal{Z}$ is a subfield of $\End^0(T)$, the degree 
$ [\mathcal{Z}:\BQ]$ is either $1$ or $g$ or $2g$. If $ [\mathcal{Z}:\BQ] >1$ then
$2g$ is divisible by 
$$d^2 \cdot [\mathcal{Z}:\BQ] \ge 2^2 g=4g,$$ which is nonsense.
Hence, $ [\mathcal{Z}:\BQ]=1$, i.e., $\mathcal{Z}=\BQ$ and $\End^0(T)$ is a central  division $\BQ$-algebra
of dimension $d^2$ with $d^2 |2g$.   Then every maximal subfield $E$ of  the central division $\BQ$-algebra $\End^0(T)$ has degree $d$ over $\BQ$ \cite[Sect. 13.1, Cor. b]{Pierce}.
By  the already proven assertion (ii),  $d\in \{1,g,2g\}$. Since $d > 1$, we obtain that either $d=g$ and  $g^2 \mid 2g$ or $d=2g$ and $(2g)^2\mid 2g$. This implies that 
$d=g$ and $g=1$ or $2$.
Since $g \ge 3$, we get a contradiction, which implies that $\End^0(T)$ is a field.

It follows from the already proven assertion (ii) that the degree $\dim_{\BQ}(\End^0(T))$ of the number field $\End^0(T)$  is either $1$ or $g$ or $2g$. 

Assertion (iv) is obvious and was included just for the sake of completeness.

In order to handle the structure of $\Aut(T)$, let us check first that the only roots of unity in $\End^0(T)$ are $1$ and $-1$. If this is not the case then the field 
$\End^0(T)$  contains either $\sqrt{-1}$ or a primitive $p$th root of unity $\zeta$ where $p$ is a certain odd prime. In the former case
$\End^0(T)$ contains the quadratic field $\BQ(\sqrt{-1})$, which contradicts (ii). In the latter case $\End^0(T)$  contains either
the quadratic field $\BQ(\sqrt{-p})$ or the quadratic field $\BQ(\sqrt{p})$: each of these outcomes contradicts  (ii) as well. 

Now recall that $\End(T)$ is an order in the number field $E=\End^0(T)$
and $\Aut(T) =\End(T)^{*}$ is its group of units.
 By the Theorem of Dirichlet  about units \cite[Ch. II, Sect. 4, Th. 5]{BS},
the group of units is
\begin{equation}
\label{Dunits}
\Aut(T) \cong \BZ^{d} \times \{\pm 1\} \ \text{ with } \ d=r_E+s_E-1
\end{equation}
where $r_E$ is the number of real field embeddings $E \hookrightarrow \BR$ and 
\begin{equation}
\label{realComplex}
r_E+2s_E =[E:\BQ], \ \text{ i.e., } \ s_E=\frac{[E:\BQ]-r_E}{2}.
\end{equation}

Let us prove (v). Assume that the number field $E=\End^0(T)$ has degree $2g$. A dimension argument implies that $\Lambda_{\BQ}$ is a $1$-dimensional $E$-vector space
and $V=\Lambda_{\BQ}\otimes_{\BQ}\BR$ is a free $E_{\BR}=E\otimes_{\BQ}\BR$-module of rank $1$. Hence $E_{\BR}$ coincides with its own centralizer
$\End_{E_{\BR}}(V)$  in $\End_{\BR}(V)$. Since $J$ commutes with $\End^0(T)=E$, it also commutes with $E_{\BR}$ and therefore
$$J \in \End_{E_{\BR}}(V) =E_{\BR}.$$
Recall that the $\BR$-algebra $E_{\BR}$ is isomorphic to a product of copies of $\BR$ and $\BC$.  Since $J^2=-1$,   only copies of $\BC$
appear in $E_{\BR}$, i.e., $E$ is purely imaginary, which means that
$r_E=0$ and therefore  $2g=[E:\BQ]=2s_E$.
 This  proves  the first  assertion of (v); the second one follows readily from \eqref{Dunits} combined with \eqref{realComplex}.
 
  In order to prove the
 last assertion, assume that $E$ contains two distinct proper subfields $E_1$ and $E_2$, none of which coincides with $\BQ$. Clearly,
 $$[E_1:\BQ]=g=[E_2:\BQ],$$
 which means that both field extensions $E/E_1$ and $E/E_2$ are quadratic. This implies that the (finite) automorphism group $G:=\Aut(E/\BQ)$ of the field extension $E/\BQ$
 contains two distinct elements $t_1$ and $t_2$ of order $2$ such that 
 $$E_1=\{u \in E\mid t_1(u)=u\}, \quad E_2=\{u \in E\mid t_2(u)=u\},$$
 It follows that $G$ is a  group of  order $M$ where $M$ is an even integer that is strictly greater than $2$.  Then the subfield $F:=E^G$ of $G$-invariants is a proper
 subfield of $E$ and its degree 
 $$[F:\BQ]=\frac{[E:\BQ]}{M} <\frac{2g}{2}=g.$$
 It follows from (ii) that $F=\BQ$ and therefore $M=[E:\BQ]=2g$.  
 
 If $g$ is {\sl not} a power of $2$ then there is an odd prime $p$ dividing $g$ and therefore dividing $M$. It follows that $G$ contains an element $t$ of order $p$.
 Therefore the subfield $E^t$ of $t$-invariants is a proper subfield of $E$ and its degree $[E^t:\BQ]$ is $2g/p<g$. By (ii), $E^t=\BQ$ and therefore
 $2g=[E:\BQ]=p$, which is wrong, since $p$ is odd. Hence $g$ is a power of $2$ and therefore $G$ is a finite $2$-group. It follows that $G$ has a normal subgroup $H$ of index $2$.
 Then the subfield $E^{H}$ of $H$-invariants is a proper subfield of $E$ and its degree 
 $[E^{H}:\BQ]$ equals the index $[G:H]=2$. This also contradicts (ii), which ends the proof of the last assertion of (v).
 
 Let us prove (vi). Assume that $[E:\BQ]=g$.  Then the  assertion about $\Aut(T)$ follows readily from \eqref{Dunits} combined with \eqref{realComplex}.
 If $F\ne \BQ$ is a proper subfield of $E$ then
 $$1=[\BQ:\BQ]<[F:\BQ]<[E:\BQ]=g$$
 and therefore $1<[F:\BQ]<g$, which contradicts (ii) applied to $F$ instead of $E$. So, such an $F$ does {\sl not} exist, i.e., $E$ is {\sl primitive}.
 
 Assume now that $T$ is a complex abelian variety. By Albert's classification \cite{MumfordAV},  $E=\End^0(T)$ is either a totally real number field or a CM field. If $E$ is a CM field then it contains a subfield $E_0$ of degree
 $[E:\BQ]/2=g/2$. Since $E_0$ is a subfield of $\End^0(T)$ and $1<g/2<g$ (recall that $g \ge 3$), the existence of $E_0$ contradicts  the already proven assertion (ii). This proves that $E$ is a totally real number field, i.e., $s=0, r=g$. Now the assertion about $\Aut(T)$ follows from \eqref{Dunits}.
\end{proof}

\section{Number Fields and  transitive permutation groups}
\label{almostTran}
All the results of this section are standard and pretty well known except, may be, the notion of almost double transitivity.

\begin{defn}
Let $\mathcal{T}$ be a set that consists of at least three elements. We write $\mathrm{Perm}(\mathcal{T})$ for the group of all permutations of $\mathcal{T}$.
Let $G$ be a group that acts  on $\mathcal{T}$, i.e., we are given a group homomorphism
$$G \to \mathrm{Perm}(\mathcal{T}),$$ 
whose image we denote by $\tilde{G}$, which is a subgroup of $\mathrm{Perm}(\mathcal{T})$.
We say that a {\sl transitive} action of $G$ on $\mathcal{T}$
is {\sl almost doubly transitive} if the action of $G$ on the set of all two-element subsets of $\mathcal{T}$ is transitive.
\end{defn}

\begin{rems}
\label{almostT}
\begin{enumerate}
\item
Every doubly transitive action is almost doubly transitive.
\item
Every almost doubly transitive action of $G$ on $\mathcal{T}$ is primitive, i.e., the stabilizer of a point is a maximal subgroup.
Indeed, suppose the action is {\sl not} primitive, i.e., that $\mathcal{T}$ partitions into a disjoint union of $r$ sets $\mathcal{T}_1, \dots, \mathcal{T}_r$ such that $r \ge 2$,
each $\mathcal{T}_i$ consists of $m \ge 2$ elements, and $G$ permutes $\mathcal{T}_i$s. Let $A$ be a $2$-element subset of $T_1$. Pick $b_1 \in\mathcal{T}_1$
and $b_2 \in\mathcal{T}_2$, and consider a $2$-element subset $B=\{b_1,b_2\}$ of $\mathcal{T}$. Clearly, no $s\in G$ sends $A$ to $B$, i.e., the action is not almost doubly transitive.

\item
If  $\mathcal{T}$ consists of three elements then every transitive action on $\mathcal{T}$ of any group $G$
is  almost doubly transitive.
\item
If $\mathcal{T}$ is a finite set then the group $\tilde{G}$  is a finite group of permutations of $\mathcal{T}$ that
is  primitive (resp. almost doubly transitive, resp. doubly transitive) if and only if $G$ is  primitive (resp. almost doubly transitive, resp. doubly transitive)

\item
Suppose that $\mathcal{T}$ is a finite set that consists of $n \ge 3$ elements and $G$ is a 
group that acts faithfully and almost doubly transitively on $\mathcal{T}$. Let $N$ be the order of $\tilde{G}$.

Then $N$ is divisible by $n(n-1)/2$. If $N$ is even then $\tilde{G}$ contains an element $\tilde{\sigma}$ of order $2$ and therefore there are
two distinct elements $s_1, s_2 \in \mathcal{T}$ such that
$$\tilde{\sigma}(s_1)=s_2, \tilde{\sigma}(s_2)=s_1.$$
It follows that the action of $\tilde{G}$ on $\mathcal{T}$ is doubly transitive and therefore the action of $G$ on $\mathcal{T}$ is also doubly transitive.
This implies that if either $4|n$ or $n \equiv 1 \ \bmod 4$ then the action of $G$ on $\mathcal{T}$ is doubly transitive,
because in these cases  $n(n-1)/2$ is even.
\item
Let $n=q$ be a prime power that is congruent to $3$ modulo $4$.  Let $\BF_q$ be a $q$-element finite field
and $\BF_q^{*}$  the multiplicative group of nonzero elements of $\BF_q$. Then $\BF_q^{*}$ splits into a direct product
$\BF_q^{*}=H \times \{\pm 1\}$ where $H$ is the cyclic subgroup of odd order $(q-1)/2$.  Let us put $\mathcal{\mathcal{T}}=\BF_q$ and let
$G$ be the group of affine transformations of $\BF_q$
$$ x \mapsto ax+b, \  a \in H \subset \BF_q^{*}, b \in \BF_q.$$
Then the action of $G$ on $\BF_q$ is almost doubly transitive but {\sl not} doubly transitive.
\end{enumerate}
\end{rems}

Let $\bar{\BQ}$ be the algebraic closure of $\BQ$ in $\BC$ and 
$$\Gal(\BQ)=\Gal(\bar{\BQ}/\BQ)=\Aut(\bar{\BQ}/\BQ)$$
the absolute Galois group of $\BQ$.  Let us consider the humongous group $\Aut(\BC)$ of all automorphisms of the field $\BC$. Obviously,
the subfield $\bar{\BQ}$ is $\Aut(\BC)$-invariant, which gives rise to the (restriction) homomorphism of groups
\begin{equation}
\label{CQ}
\Aut(\BC) \twoheadrightarrow \Gal(\BQ),  \quad s \mapsto \{\alpha \mapsto s(\alpha)\} \ \forall s \in \Aut(\BC), \alpha \in \bar{\BQ}
\end{equation}
which is {\sl surjective}.

Let $E$ be a number field of degree $n=[E:\BQ]$. We write    $\Sigma_E$ for the $n$-element set of  all field embeddings $\sigma: E \hookrightarrow \BC$.
 For each $\sigma \in \Sigma_E$  the image $\sigma(E)$ lies in $\bar{\BQ}$. If $t$ is an element of $\Aut(\BC)$ (or of $\Gal(\BQ)$) then the composition
$$t \circ \sigma: E \overset{\sigma}{\hookrightarrow} \bar{\BQ} \overset{t}{\to} \bar{\BQ}\subset \BC$$
also lies in $\Sigma_E$. Then the map
\begin{equation}
\label{GQE}
\Aut(\BC)  \times \Sigma_E \to \Sigma_E, \quad (t,\sigma) \mapsto t \circ \sigma
\end{equation}
is a {\sl transitive group action} of $\Aut(\BC)$ on  $\Sigma_E$, which factors through $\Gal(\BQ)$ via \eqref{CQ}. 
This action is {\sl primitive} (i.e., the stabilizer of a point is a maximal subgroup) if and only if $E$ is a primitive number field.
Similarly, the map
\begin{equation}
\label{GQEbar}
\Gal(\BQ)  \times \Sigma_E \to \Sigma_E, \quad (t,\sigma) \mapsto t \circ \sigma
\end{equation}
is a {\sl transitive group action} of $\Gal(\BQ)$ on  $\Sigma_E$, which is primitive if and only if $E$ is a primitive number field.
\begin{defn}
We say that $E$ is a {\sl doubly transitive} (respectfully {\sl almost doubly transitive}) number field if the action \eqref{GQE} (or equivalently
the action \eqref{GQEbar})
is  doubly transitive (respectfully almost doubly transitive). The corresponding finite subgroup
 $\tilde{G}=\widetilde{\Aut(\BC)}=\widetilde{\Gal(\BQ)}$ of $\mathrm{Perm}(\Sigma_E)$ is isomorphic to the Galois group $\Gal(\tilde{E}/\BQ)$ where $\tilde{E}$ is a {\sl normal closure} of $E$.
 \end{defn}
 \begin{rem}
 \label{EisoF}
 Clearly, if $E$ and $F$ are isomorphic number fields then   $E$ is  primitive (resp.  doubly transitive) (resp.  almost doubly transitive)  if 
and only if $F$ is primitive  (resp. doubly transitive) (resp. almost doubly transitive).
\end{rem}
 
 \begin{ex}
 \label{Ef}
 \begin{itemize}
 \item[(i)]
 Let $f(x) \in \BQ[x]$ be an irreducible polynomial of degree $n \ge 2$ over $\BQ$ and
 $E_f=\BQ[x]/f(x)\BQ[x]$ the corresponding number field of degree $n$. We write $\RR_f$ for the $n$-element set of roots of $f(x)$ in $\bar{\BQ}$
 and $\BQ(\RR_f)$ for the subfield of $\bar{\BQ}$ generated by $\RR_f$. By definition, $\BQ(\RR_f)$  is a splitting field of $f(x)$ that is a finite Galois extension
 of $\BQ$. We write $\Gal(f)$ for the Galois group $\Gal(\BQ(\RR_f)/\BQ)$ of the field extension $\BQ(\RR_f)/\BQ$. It is well known that 
 $\Gal(\BQ)$ acts transitively (through  $\Gal(f)$ ) on $\RR_f$.  There is a $\Gal(\BQ)$-equivariant bijection between $\Sigma_{E_f}$ and $\RR_f$ that is defined as follows.
 To each $\alpha \in \RR_f$ corresponds the field embedding
 $$\sigma_{\alpha}: E_f=\BQ[x]/f(x)\BQ[x] \hookrightarrow \bar{\BQ}\subset \BC, \ h(x)+f(x)\BQ[x] \mapsto h(\alpha)$$
(in particular, the coset of $x$ goes to $\alpha$). This implies that the field $E_f$ is  doubly transitive (respectfully almost doubly transitive) if and only if
 the action of $\Gal(f)$ on $\RR_f$ is  doubly transitive (respectfully almost doubly transitive). The similar characterization of primitive number fields is well known:
 
{\sl  the field $E_f$ is primitive if and only if the action of $\Gal(f)$ on $\RR_f$ is primitive}.

On the other hand, obviously,
$r_{E_f}$ equals the number of {\sl real roots} of $f(x)$ and
$2 s_{E_f}=n-r_{E_f}$.
 \item[(ii)]
  Conversely, let $F$ be a number field of degree $n$ and $z \in F$ be a primitive element of $F$, i.e., the smallest subfield $\BQ(z)$ of $F$ that contains $z$ coincides with $F$
 (such an element always exists). Let $f(x)\in\BQ[x]$ be the minimal polynomial of $z$, i.e., $f(x)$ is irreducible over $\BQ$ and $f(z)=0$; in addition, $\deg(f)=n$. Then there is a field
 isomorphism $E_f\cong F$ such that the coset $x+f(x)\BQ[x]\in E_f$ goes to $z \in F$.  Therefore the number field $F$ is (almost) doubly transitive if and only if $\Gal(f)$ acts (almost) doubly transitively on $\RR_f$.
 \end{itemize}
 \end{ex}

\begin{thm}
\label{existReal}
Let $n \ge 3$ be an integer. Let $\mathbf{r},\mathbf{s}$ be nonnegative integers such that
\begin{equation}
\label{rTwos}
\mathbf{r}+2\mathbf{s}=n.
\end{equation}
Then there exists a number field $E$ of degree $n$ that enjoys the following properties.

\begin{itemize}
\item[(i)]
$r_E=\mathbf{r}, \quad s_E=\mathbf{s}$.
\item[(ii)]
$E$ is doubly transitive.
\end{itemize}
\end{thm}

\begin{proof}
We will use  weak approximation in $\BQ$, approximating several polynomials in $\BQ[x]$ with respect to several metrics in $\BQ$.

First, fix a degree $n$ monic polynomial $h_{\infty}(x) \in \BZ[x]$ that has precisely $\mathbf{r}$ distinct real roots
and $\mathbf{s}$  distinct pairs of non-real complex-conjugate roots. (E.g., one may take
$$h_{\infty}(x) =\left(\prod_{i=1}^{\mathbf{r}} (x-i)\right)\cdot \left( \prod_{j=1}^{\mathbf{s}} (x^2+j^2)\right)\in \BZ[x]\subset \BQ[x].)$$
Second,  take any prime $p$ and choose a   monic ($p$-adic Eisenstein polynomial) $h_p(x) \in \BZ[x]$, all whose coefficients (except the leading one)
are divisible by $p$ while the constant term is {\sl not} divisible by $p^2$. (E.g., one may take
$$h_p(x)=x^n-p\in \BZ[x]\subset \BQ[x].)$$
Third, take any prime $\ell \ne p$ and choose  a monic irreducible  polynomial  $$\tilde{u}_{\ell}(x) \in \BF_{\ell}[x]$$ over $\BF_{\ell}$ of degree $n-1$. (Such a polynomial always exists for any given $\ell$ and $n$.) Since $n-1 \ge 3-1>1$, the irreducibility of $\tilde{u}_{\ell}(x)$ implies that
$$\tilde{u}_{\ell}(0) \ne 0.$$
Let $u_{\ell}(x)\in \BZ[x]$ be any monic degree $(n-1)$ polynomial with integer coefficients, whose reduction modulo $\ell$ coincides with $\tilde{u}_{\ell}(x)$.
Let us put
$$h_{\ell}(x):=x \cdot u_{\ell}(x)\in \BZ[x]\subset \BQ[x].$$

It follows from  a weak approximation theorem  \cite[Th. 1]{Artin} that there is a monic degree $n$ polynomial $f(x) \in \BQ[x]$ that enjoys the following properties.

\begin{itemize}
\item[(a)]
$f(x)$ is so close to $h_{\infty}(x)$ in the archimedean topology that it also  has precisely $\mathbf{r}$ distinct real roots
and $\mathbf{s}$  distinct pairs of non-real complex-conjugate roots. 
\item[(b)]
$f(x)-h_p(x) \in p^2 \cdot \BZ_p[x]$. This implies that $f(x)$ is irreducible over the field $\BQ_p$ of $p$-adic numbers and therefore is irreducible over $\BQ$.
\item[(c)]
$f(x)-h_{\ell}(x) \in \ell \cdot  \BZ_{\ell}[x]$.  This implies that 
$$f(x) \in \BZ_{\ell}[x], \ f(x) \bmod \ell=x \cdot \tilde{u}_{\ell}(x) \in \BF_{\ell}[x].$$
By Hensel's Lemma, there are $$\alpha \in \ell \BZ_{\ell}\subset \BZ_{\ell}$$ and a monic degree $(n-1)$ polynomial $v(x) \in \BZ_{\ell}[x]$ such that
\begin{equation}
\label{vanDer}
f(x)=(x-\alpha) v(x)\in \BZ_{\ell}[x], \quad  v(x) \bmod \ell=\tilde{u}_{\ell}(x) \in \BF_{\ell}[x].
\end{equation}
\end{itemize}
By \cite[Sect. 66]{Waerden}, the irreducibility of $\tilde{u}_{\ell}(x)$ combined with \eqref{vanDer} imply
that $\Gal(f)$, viewed as the certain transitive subgroup of $\Perm(\mathfrak{R}_f)$, contains a permutation $s$ that is a cycle of length $n-1$.  In particular,
if $\alpha\in \RR_f$ is the fixed point of $s$ then the cyclic subgroup $<s>$ of $\Gal(f)$ generated by $s$ acts transitively on $\RR_f \setminus \{\alpha\}$.
Now the transitivity
of $\Gal(f)$ implies its double transitivity. It remains only to put $E:=E_f$
and apply the results of Example \ref{Ef}(i).

\end{proof}

\section{Hodge groups}
\label{HDG}
Recall that $\Lambda_{\BR}=V$ carries the natural structure of a complex vector space. This gives rise to the injective homomorphism of real Lie groups
$$h: \BC^{*} \hookrightarrow \Aut_{\BR}(\Lambda_{\BR})$$
where $h(z)$ is multiplication by a nonzero complex number $z$ in $\Lambda_{\BR}=V$.  Let
$\mathbb{S}^1\subset \BC^{*}$ be the subgroup of all complex numbers $z$ with $|z|=1$.
Clearly, $h(\mathbb{S}^1)$ is a one-dimensional closed connected real  Lie subgroup of 
$\Aut_{\BR}(\Lambda_{\BR})$; in addition,  the Lie algebra of  $h(\mathbb{S}^1)$ is
$$\BR \cdot J  \subset \End_{\BR}(\Lambda_{\BR}).$$
Actually,
$h(\mathbb{S}^1)$ lies in the special linear group $\mathrm{SL}(\Lambda_{\BR})$ while
$\BR \cdot J$ lies in the Lie algebra  $\sll(\Lambda_{\BR})$ of traceless operators in $\Lambda_{\BR}$.

By definition \cite{Mumford,Serre} (see also \cite{Z84}), the Hodge group $\mathrm{Hdg}(T)$ of the rational Hodge structure $\mathrm{H}_1(T,\BQ)=\Lambda_{\BQ}$ is the smallest algebraic $\BQ$-subgroup $G$ of $\mathrm{GL}(\Lambda_{\BQ})$, whose group of real points
$$G(\BR) \subset \Aut_{\BR}(\Lambda_{\BR})$$ contains $h(\mathbb{S}^1)$. One may easily check that $\mathrm{Hdg}(T)$ enjoys the following properties that we will freely use throughout the text.

\begin{itemize}
\item[(i)]
$\mathrm{Hdg}(T)$ is a {\sl connected} algebraic $\BQ$-group that is a subgroup of the special linear group
$\SL(\Lambda_{\BQ})$.
\item[(ii)]
The centralizer of $\mathrm{Hdg}(T)$ in $\End_{\BQ}(\Lambda_{\BQ})$ coincides with $\End^0(T)$.
\item[(iii)]
A $\BQ$-vector subspace $H_{\BQ}$ of $\Lambda_{\BQ}$ is $\mathrm{Hdg}(T)$-invariant if and only if it is a rational Hodge substructure of $\Lambda_{\BQ}$.
\item[(iv)]
The subspace of $\mathrm{Hdg}(T)$-invariants
$$\mathrm{H}^2(T,\BQ)^{\mathrm{Hdg}(T)}\subset \mathrm{H}^2(T,\BQ)=\Hom_{\BQ}(\wedge_{\BQ}^2\Lambda_{\BQ},\BQ)$$ 
coincides with the subspace $\mathrm{H}^2(T,\BQ)\cap H^{1,1}(T)$ of two-dimensional Hodge classes on $T$.
\item[(v)] The group $\Hdg(T)(\BQ)$  of $\BQ$-points is {\sl Zariski dense} in $\Hdg(T)$, because $\Hdg$ is connected and the field  $\BQ$ in  infinite (see \cite[Cor. 18.3]{Borel}).
\item[(vi)] If $T \ne \{0\}$ then $\mathrm{Hdg}(T)$ is a positive-dimensional connected algebraic group.
\end{itemize}

Let us consider the $\BQ$-Lie algebra
$\mathrm{hdg}_{T}$
of the linear  {\sl algebraic} $\BQ$-group $\mathrm{Hdg}(T)$.  By definition, $\mathrm{hdg}_{T}$ is a linear {\sl algebraic} Lie subalgebra of
$\End_{\BQ}(\Lambda_{\BQ})$. 

\begin{rem}
\label{minLieLie}
Clearly, $\mathrm{hdg}_{T}$ is the {\sl smallest algebraic} Lie $\BQ$-subalgebra $\mathfrak{g}$ of
$\End_{\BQ}(\Lambda_{\BQ})$ such that
\begin{equation}
\label{minLieJ}
J \in \mathfrak{g}\otimes_{\BQ}\BR.
\end{equation}
\end{rem}

 Properties (i) and (ii) above imply that 
\begin{equation}
\label{hdgSL0}
\mathrm{hdg}_{T} \subset \sll(\Lambda_{\BQ})\subset \End_{\BQ}(\Lambda_{\BQ})
\end{equation}
and the centralizer of $\mathrm{hdg}_{T}$ in $\End_{\BQ}(\Lambda_{\BQ})$ is described as follows.
\begin{equation}
\label{hdgC}
\End_{\mathrm{hdg}_{T} }(\Lambda_{\BQ})=\End^0(T).
\end{equation}

Clearly, 
$$J\in \mathrm{hdg}_{T,\BR}:=\mathrm{hdg}_{T}\otimes_{\BQ}\BR\subset \End_{\BQ}(\Lambda_{\BQ})\otimes\BR=\End_{\BR}(\Lambda_{\BR}).$$
Let us consider the {\sl complexification}
$$\mathrm{hdg}_{T,\BC}:=\mathrm{hdg}_{T}\otimes_{\BQ}\BC\subset \End_{\BQ}(\Lambda_{\BQ})\otimes\BC=\End_{\BC}(\Lambda_{\BC})$$
where
$$\Lambda_{\BC}=\Lambda_{\BQ}\otimes_{\BQ}\BC=\Lambda_{\BR}\otimes_{\BR}\BC.$$
We have
$$J\in \mathrm{hdg}_{T,\BR}=\mathrm{hdg}_{T,\BR}\otimes 1\subset \mathrm{hdg}_{T,\BR}\otimes_{\BR}\BC=\mathrm{hdg}_{T,\BC}\subset \End_{\BC}(\Lambda_{\BC}).$$
(See \eqref{notationAbuse} and Remark \ref{notAB}.) In what follows, we will write $J$ instead of 
 $J_{\BC}$, slightly abusing notation.

The   group $\Aut(\BC)$  acts naturally,  {\sl semi-linearly} and compatibly on $\Lambda_{\BC}$, 
$\End_{\BC}(\Lambda_{\BC})$ and $\mathrm{hdg}_{T,\BC}$.

The {\sl minimality property} of $\mathrm{Hdg}(T)$ allows us to give the following ``explicit'' description
of the complexification
$\mathrm{hdg}_{T,\BC}$ (compare with \cite[Lemma 6.3.1]{Z84}).

\begin{thm}
\label{minLie}
The complex Lie algebra $\mathrm{hdg}_{T,\BC}$  coincides with the Lie subalgebra $\mathfrak{u}$ of $ \End_{\BC}(\Lambda_{\BC})$
generated by all $s(J)$ where $s$ run over the group $\Aut(\BC)$. In particular, $\mathrm{hdg}_T$ coincides with the smallest $\BQ$-Lie subalgebra
$\mathfrak{g}\subset \End_{\BQ}(\Lambda_{\BQ})$ such that 
$$\mathfrak{g}\otimes_{\BQ}\BR\subset \End_{\BQ}(\Lambda_{\BQ})\otimes_{\BQ}\BR=\End_{\BR}(\Lambda_{\BR})$$
contains $J$.
\end{thm}

\begin{proof}
Clearly, $\mathfrak{u} \subset \mathrm{hdg}_{T,\BC}$. 
Let us prove that $\mathfrak{u}$ is an {\sl algebraic} complex Lie subalgebra of $\End_{\BC}(\Lambda_{\BC})$.

 Recall that
\begin{equation}
\label{Jsquare0}
J \in \End_{\BR}(\Lambda_{\BR})\subset \End_{\BC}(\Lambda_{\BC}); \ J^2=-1.
\end{equation}

Clearly, $J: \Lambda_{\BC} \to \Lambda_{\BC}$ is a semisimple $\BC$-linear operator, whose spectrum
consists of eigenvalues, $\mathbf{i}$ and $-\mathbf{i}$, because $J^2=-1$.
Similarly,  for all $s \in \Aut(\BC)$ the
$\BC$-linear operator $s(J):  \Lambda_{\BC} \to  \Lambda_{\BC}$ is also semisimple and its
spectrum is also $\{\mathbf{i},-\mathbf{i}\}$, because (in light of  \eqref{Jsquare0})
\begin{equation}
\label{Jsquare}
s(J)^2=s(J^2)=s(-1)=-1.
\end{equation}
 It follows that the $\BQ$-vector subspace
$\BQ(s(J))$ of $\BC$ generated by the {\sl spectrum} of $s(J)$ coincides with $\BQ\cdot \mathbf{i}$; in particular, the $\BQ$-vector (sub)space
$\BQ(s(J))$  is one-dimensional.  This implies that each $\BC\cdot s(J)$ is an algebraic
$\BC$-Lie subalgebra of $\End_{\BC}(\Lambda_{\BC})$, because each {\sl replica} of $s(J)$
is a scalar multiple of $s(J)$. Thus, the linear $\BC$-Lie algebra $\mathfrak{u}$  is generated by the  algebraic Lie subalgebras
$\BC\cdot \sigma(f)$ and therefore is algebraic itself, thanks to  \cite[volume 2, Ch. 2, Sect. 14, Th. 14]{Chevalley}.
Clearly, $\mathfrak{u}$ is $\Aut(\BC)$-invariant.
By Lemma \ref{subspace} (see below), 
 $\mathfrak{u}$ is
defined over $\BQ$, i.e., there is 
a $\BQ$-vector subspace
$$\mathfrak{u}_0 \subset \End_{\BQ}(\Lambda_{\BQ})$$
such that
$$\mathfrak{u}=\mathfrak{u}_0\otimes_{\BQ}\BC,   \quad \mathfrak{u}_0=\mathfrak{u} \cap   \End_{\BQ}(\Lambda_{\BQ}).$$
Since $\mathfrak{u}$ is a $\BC$-Lie subalgebra, the latter equality implies that $\mathfrak{u}_0$ is a $\BQ$-Lie subalgebra of  $\End_{\BQ}(\Lambda_{\BQ})$, and this Lie algebra is algebraic, in light of
\cite[volume 2, Ch. 2, Sect. 14, Prop. 4]{Chevalley}.
Clearly, 
$$\mathfrak{u}=\left(\mathfrak{u}_0\otimes_{\BQ}\BR\right)\oplus \mathbf{i} \cdot
\left(\mathfrak{u}_0\otimes_{\BQ}\BR\right) $$
as a real vector  space. This implies that
\begin{equation}
\label{QCR}
\mathfrak{u}_0\otimes_{\BQ}\BR=\mathfrak{u} \cap \End_{\BR}(\Lambda_{\BR}).
\end{equation}
Let $\mathcal{U}$ be the connected algebraic $\BQ$-subgroup of $\GL(\Lambda_{\BQ})$, whose Lie algebra coincides
with $\mathfrak{u}_0$. We need to prove that 
$$\mathfrak{u}_0=\mathrm{hdg}_T.$$
 Clearly, $\mathfrak{u}_0\subset \mathrm{hdg}_T$, because the comlexification of $\mathfrak{u}_0$
lies in the complexification of $\mathrm{hdg}_T$. We know that
$$J \in \mathfrak{u}_0\otimes_{\BQ}\BC=\mathfrak{u}.$$
 Since $J \in \End_{\BR}(\Lambda_{\BR})$, it follows from \eqref{QCR} that 
$$J\in \mathfrak{u}_0\otimes_{\BQ}\BR.$$
In light of Remark \ref{minLieLie}, $\mathfrak{u}_0\supset \mathrm{hdg}_T$.
This implies that $\mathfrak{u}_0=\mathrm{hdg}_T$, which ends the proof.
\end{proof}

\begin{cor}
\label{JtoF}
Let us put
$$f_T:=\frac{1}{\mathbf{i}}J \in \End_{\BC}(\Lambda_{\BC}).$$
Then
\begin{equation}
s(f_T)^2=1 \ \forall s\in \Aut(\BC)
\end{equation}
and the complex Lie algebra $\mathrm{hdg}_{T,\BC}$  coincides with the Lie subalgebra $\mathfrak{u}$ of $ \End_{\BC}(\Lambda_{\BC})$
generated by all $s(f_T)$ where $s$ run over the group $\Aut(\BC)$. In particular, $\mathrm{hdg}_T$ coincides with the smallest $\BQ$-Lie subalgebra
$\mathfrak{g}\subset \End_{\BQ}(\Lambda_{\BQ})$ such that 
$$\mathfrak{g}\otimes_{\BQ}\BC\subset \End_{\BQ}(\Lambda_{\BQ})\otimes_{\BQ}\BC=\End_{\BC}(\Lambda_{\BC})$$
contains $f_T$.
\end{cor}

\begin{proof} Since $\mathbf{i}=\sqrt{-1}$,
\begin{equation}
\label{JiF}
s(\mathbf{i})=\pm \mathbf{i}, s(f_T)=\pm\mathbf{i}\cdot  s(J) \ \forall s \in \Aut(\BC).
\end{equation}
Therefore 
$$s(f_T)^2=\left(\pm\mathbf{i}\cdot  s(J)\right)^2=-s(J)^2.$$
It follows from \eqref{Jsquare} that
$$s(f_T)^2= - (-1)=1,$$
which proves our first assertion. It follows from \eqref{JiF} that
$$\BC \cdot s(f_T)=\BC \cdot  s(J) \ \forall s \in \Aut(\BC).$$
Now our second assertion follows from Theorem \ref{minLie}.

\end{proof}

The following definition was actually introduced in  \cite[Definition 1.1]{Z91}.

\begin{defn}

Let $C$ be a field, $k$ a subfield of $C$, and $\mathfrak{g}$ a $k$-Lie algebra. Let $\mathfrak{g}_C$ be the corresponding $C$-Lie algebra $\mathfrak{g}\otimes_k C$. 

If $f \in \mathfrak{g}_C$ then the {\sl Hodge ideal} attached to $f$ is the {\sl smallest ideal} 
$\mathrm{hmt}=\mathrm{hmt}(f)$ 
\footnote{Here $\mathrm{hmt}$ is short for Hodge-Mumford-Tate.}
of  $\mathfrak{g}$
such that
$$\mathrm{hmt}_C:=\mathrm{hmt}\otimes_k C$$
contains $f$.  We call $f$ a {\sl Hodge element} of $\mathfrak{g}$ if $\mathrm{hmt}(f)=\mathfrak{g}$.
\end{defn}

\begin{rem}
\label{HodgeElement}
Let $\mathfrak{a}$ be the {\sl smallest ideal} of $\mathrm{hdg}_{T}$ such that
$\mathfrak{a}_{\BC}:=\mathfrak{a}\otimes_{\BQ}\BC$ contains $f_T$.  Clearly, 
$\mathfrak{a}_{\BC}$ contains $s(f_T)$ for all $s\in \Aut(\BC)$. It follows from Corollary \ref{JtoF}
that $\mathfrak{a}_{\BC}=\mathrm{hdg}_{T,\BC}$. This implies that 
$$\mathfrak{a}=\mathrm{hdg}_{T}.$$
The latter equality means that $f_T$ is a {\sl Hodge element} of the $\BQ$-Lie algebra $\mathrm{hdg}_{T}$
 where
$$k=\BQ,  \quad C=\BC.$$
\end{rem}

For the sake of simplicity, from now on let us assume that $T$ is {\sl simple}.  This means that the natural faithful representation of $\mathrm{Hdg}(T)$ in 
$$\Lambda_{\BQ}=\mathrm{H}_1(T,\BQ)$$
 is {\sl irreducible} and therefore $\End^0(T)$ is a division algebra over $\BQ$. This implies that the $\BQ$-algebraic (sub)group  $\mathrm{Hdg}(T)$  is {\sl reductive}. In addition, the $\BQ$-Lie (sub)algebra
$$\mathrm{hdg}_{T} \subset \sll\left(\mathrm{H}_1(T,\BQ)\right)\subset
\End_{\BQ}\left(\mathrm{H}_1(T,\BQ)\right)$$
is {\sl reductive algebraic}, the faithful $\mathrm{hdg}_{T}$-module $\mathrm{H}_1(T,\BQ)$ is simple and the centralizer of $\mathrm{hdg}_{T}$ in $\End_{\BQ}\left(\mathrm{H}_1(T,\BQ)\right)$ is the division $\BQ$-algebra $\End^0(T)$. Then the {\sl center} $\mathcal{Z}(T)$ of  $\End^0(T)$ is a {number field}.

Let us split the {\sl reductive} $\BQ$-Lie algebra $\mathrm{hdg}_{T} $ into a direct sum
$$\mathrm{hdg}_{T} =\mathrm{hdg}_{T}^{\mathrm{ss}} \oplus \mathfrak{c}_T$$
of the semisimple $\BQ$-Lie algebra 
$$\mathrm{hdg}_{T}^{\mathrm{ss}}=[\mathrm{hdg}_T ,\mathrm{hdg}_{T} ]$$
and the {\sl center} $\mathfrak{c}_{T}$ of $\mathrm{hdg}_{T}$ with
$$\mathfrak{c}_{T}\subset \mathcal{Z}(T)\subset \End^0(T).$$

The following useful assertion  is well known in the case of abelian varieties.

\begin{lem}
\label{center}
Suppose that $T$ is simple and  $\mathcal{Z}(T)=\BQ$ (e.g., $\End^0(T)=\BQ$).  Then $\mathfrak{c}_{T}=\{0\}$, i.e., the $\BQ$-Lie algebra is semisimple and therefore $\mathrm{Hdg}(T)$ is a semisimple $\BQ$-algebraic group.
\end{lem}

\begin{proof}
The result follows readily from the combination of  inclusions
$$\mathfrak{c}_{T}\subset \mathcal{Z}(T)=\BQ, \
\mathfrak{c}_{T}\subset \mathrm{hdg}_{T} \subset \sll\left(\mathrm{H}_1(T,\BQ)\right).$$
\end{proof}

The next example deals with the opposite case when the endomorphism algebra of a simple torus $T$ is a number field of (largest possible) degree $2\dim(T)$.

 \begin{ex}
 \label{CMtorusT}
 Suppose that a complex torus $T=V/\Lambda$ is simple of dimension $g$ and $\End^0(T)$ is a number field $E$ of degree $2g$. Then $\Lambda_{\BQ}$ becomes a one-dimensional vector space over $E$. Therefore
 $$E=\End_E(\Lambda_{\BQ}), \ E^{*}=\Aut_E(\Lambda_{\BQ}).$$
  This implies that 
 \begin{equation}
 \label{CMnorm}
 \Hdg(T)(\BQ) \subset E^{(1)}:=\{e\in E^{*}\mid \mathrm{Norm}_{E/\BQ}(e)=1\}\subset E^{*}=\Aut_E(\Lambda_{\BQ}).
 \end{equation}
 Here
 $$\mathrm{Norm}_{E/\BQ}: E^{*} \to \BQ^{*},  \ e \mapsto \prod_{\sigma\in \Sigma_E}\sigma(e)$$
 is the {\sl norm homomorphism} of the multiplicative groups of fields attached to the field extension $E/\BQ$
 
 Let $\mathcal{S}_{E}=\mathrm{Res}_{E/\BQ}(\mathbb{G}_m)$ be the $2g$-dimensional {\sl algebraic} torus over $\BQ$ obtained from the multiplicative group $\mathbb{G}_m$ by the Weil's restriction of scalars from $E$ to $\BQ$.  
  Then $\mathcal{S}_{E}(\BQ)=E^{*}$ and for each $\sigma \in \Sigma_E$ there is a
 certain character
 $$\delta_{\sigma}: \bar{\mathcal{S}}_{E}:=\mathcal{S}_{E}\times_{\BQ} \bar{\BQ} \to \mathbb{G}_m \times_{\BQ} \bar{\BQ}$$
 of the algebraic torus $\bar{\mathcal{S}_{E}}$ over $\bar{\BQ}$ 
 such that the restriction of $\delta_{\sigma}$ to
 $$E^{*}=\mathcal{S}_{E}(\BQ)\subset \mathcal{S}_{E}(\bar{\BQ})=\bar{\mathcal{S}}_{E}(\bar{\BQ})$$
 coincides with
 $$\sigma: E^{*}\hookrightarrow \bar{\BQ}^{*}.$$
 In addition, the $2g$-element set $\{\delta_{\sigma}\mid \sigma\in \Sigma_K\}$ constitutes a {\sl basis} of the free $\BZ$-module 
  $\mathrm{X}(\bar{\mathcal{S}}_{E})$ of {\sl characters of the algebraic torus} $\bar{\mathcal{S}}_{E}$ over $\bar{\BQ}$.
  Since $\mathcal{S}_{E}$ is defined over $\BQ$, the group  $\mathrm{X}(\bar{\mathcal{S}}_{E})$  is
  provided with the natural structure of a $\Gal(\BQ)$-module in such a way that
 \begin{equation}
 \label{GactionX}
 s(\delta_{\sigma})=\delta_{s(\sigma)} \ \forall \sigma\in \Sigma_E, s \in \Gal(\BQ)
 \end{equation}
(\cite[Ch. II, Sect. 1]{SerreAbelian}, \cite[Ch. III, Sect. 5 and 6]{Vos}).
Clearly, the character 
$$\chi=\prod_{\sigma\in\Sigma_E}\delta_{\sigma} \in \mathrm{X}(\bar{\mathcal{S}}_{E})$$
is  $\Gal(\BQ)$-invariant and may be viewed as the character of $\mathcal{S}_{E}$ such that
$$\chi(e)=\mathrm{Norm}_{E/\BQ}(e)\in \BQ^{*} \quad \forall e \in E^{*}=\mathcal{S}_{E}(\BQ).$$
Let us put
$$\mathcal{S}_{E}^{1}=\ker(\chi).$$
Since $\chi$ is obviously non-divisible in the group of characters, $\mathcal{S}_{E}^{1}$ is an algebraic $\BQ$-subtorus of dimension $2g-1$ in $\mathcal{S}_{E}$ such that
\begin{equation}
\label{SE1BQ}
\mathcal{S}_{E}^{1}(\BQ)=\ker(\mathrm{Norm}_{E/\BQ})=E^{(1)}.
\end{equation}
Combining \eqref{CMnorm} and \eqref{SE1BQ}, and taking into account that  $\Hdg(T)(\BQ)$ is Zariski dense in $\Hdg(T)$, we conclude that
\begin{equation}
\label{Hdg1E}
\Hdg(T)\subset \mathcal{S}_{E}^{1}.
\end{equation}
In particular, if $\mathcal{S}_{E}^{1}$ is a {\sl simple} algebraic torus over $\BQ$ then 
$$\Hdg(T)= \mathcal{S}_{E}^{1}.$$

By definition of $\mathcal{S}_{E}^{1}$, the Galois module   $\mathrm{X}(\bar{\mathcal{S}}_{E}^1)$
of characters of the algebraic $\bar{\BQ}$-torus 
$$\bar{\mathcal{S}}_{E}^1=\mathcal{S}_{E}^{1}\times_{\BQ}\bar{\BQ}$$
 is the quotient 
 $\mathrm{X}(\bar{\mathcal{S}}_{E})/(\BZ\cdot \chi)$. It follows from \eqref{GactionX} that
 the $\Gal(\BQ)$-module $\mathrm{X}(\bar{\mathcal{S}}_{E}^1)$ is isomorphic to the quotient
 $\BZ^{\Sigma_E}/\BZ\cdot {\mathbf 1}$ where $\BZ^{\Sigma_E}$ is the free $\BZ$-module
 of functions $\phi: \Sigma_E \to \BZ$ and $\mathbf{1}$ is the constant function $1$. It follows easily that
 the Galois module  $\mathrm{X}(\bar{\mathcal{S}}_{E}^1)\otimes\BQ$ is isomorphic to the $\BQ$-vector space 
 $$\big(\BQ^{\Sigma_E}\big)^0:=\{\phi: \Sigma_E \to \BQ\mid \sum_{\sigma\in \Sigma_E}\phi(\sigma)=0\}$$
of $\BQ$-valued functions on $\Sigma_E$ with zero ``integral''.  Recall that the action of $\Gal(E)$ on $\Sigma_E$
is transitive and this action induces the structure of the Galois module on $\big(\BQ^{\Sigma_E}\big)^0$. Notice
that if the action of $\Gal(\BQ)$ on  $\big(\BQ^{\Sigma_E}\big)^0$ is {\sl doubly transitive} (i.e., $E$ is doubly transitive) then the representation
of $\Gal(\BQ)$ in  $\big(\BQ^{\Sigma_E}\big)^0$ is {\sl irreducible}, i.e., the Galois module $\mathrm{X}(\bar{\mathcal{S}}_{E}^1)\otimes\BQ$ is 
simple, which means that the algebraic $\BQ$-torus  $\mathcal{S}_{E}^{1}$ is simple and therefore $\Hdg(T)= \mathcal{S}_{E}^{1}$. So we have proven that
\begin{equation}
\label{simpleDouble}
\Hdg(T)= \mathcal{S}_{E}^{1} 
\end{equation}
if $E$ is {\sl doubly transitive}. In particular, the algebraic $\BQ$-torus is simple.
\end{ex}

\begin{thm}
\label{Double2simple}
Let $T=V/\Lambda$ be a simple complex torus of dimension $g>2$ such that its endomorphism algebra is a number field $E$ of degree $2g$ that 
is doubly transitive.

Then:

\begin{itemize}
\item[(i)]
The Hodge group $\Hdg(T)$ of $T$ coincides with  $\mathcal{S}_{E}^{1}$. In addition,  $\Hdg(T)$  is a simple algebraic $\BQ$-torus of dimension $2g-1$.
\item[(ii)]
The $\Hdg(T)$-module $\mathrm{H}^2(T,\BQ)$ is simple. In particular, $T$ is $2$-simple.

\end{itemize}

\end{thm}

\begin{proof}
We keep the notation of Example \ref{CMtorusT} where the assertion (i) and the simplicity of $\Hdg(T)$ are already proven.

In order to prove (ii), notice that $\mathrm{H}^2(T,\BQ)=\Hom_{\BQ}(\wedge^2_{\BQ}\Lambda_{\BQ},\BQ)$, so, it suffices to check that the $\Hdg(T)$-module 
$\wedge^2_{\BQ}\Lambda_{\BQ}$ is simple.

If $\sigma \in \Sigma_E$ then let us consider the character  $\delta_{\sigma}^{(1)}$ of $\bar{\mathcal{S}}_{E}^1$ that is the restriction of the character $\delta_{\sigma}$ to 
$\bar{\mathcal{S}}_{E}^1$. Clearly, 
$$\prod_{\sigma\in \Sigma_E}\delta_{\sigma}^{(1)}=1 \in \mathrm{X}(\bar{\mathcal{S}}_{E}^1)$$
and this is the only ``nontrivial'' multiplicative relation between $\delta_{\sigma}^{(1)}$. In particular, if $A$ and $B$ are two {\sl distinct} $2$-element subsets of $\Sigma_E$ then
\begin{equation}
\label{AneB}
\delta_A^{1}:=\prod_{\sigma\in A}\delta_{\sigma}^{(1)} \ne \prod_{\sigma\in B}\delta_{\sigma}^{(1)}=:\delta_B^{1}.
\end{equation}
In other words, $\delta_A^{1}$ and $\delta_B^{1}$ are {\sl distinct} characters of $\bar{\mathcal{S}}_{E}^1$.

Let us fix an order on the $2g$-element set $\Sigma_E$ and consider the $\bar{\BQ}$-vector space
$$\bar{\Lambda}=\Lambda_{\BQ}\otimes_{\BQ}\bar{\BQ},$$
which is endowed with the natural faithful action of  $\bar{\mathcal{S}}_{E}^1$, and splits into a direct sum
$$\bar{\Lambda}=\oplus_{\sigma \in \Sigma_E}\bar{\Lambda}_{\sigma}$$
of {\sl one-dimensiona}l weight subspaces $\bar{\Lambda}_{\sigma}$ defined by the condition that 
$\bar{\mathcal{S}}_{E}^1$ acts on $\bar{\Lambda}_{\sigma}$ by the character $\delta_{\sigma}^{(1)}$.

We have
$$\wedge^2_{\BQ}\Lambda_{\BQ}\otimes_{\BQ}\bar{\BQ}=\wedge^2_{\bar{\BQ}}\bar{\Lambda}=
\oplus_{A}\bar{\Lambda}_A,$$
where  
$$A=\{\sigma,\tau\} \subset  \Sigma_E; \ \sigma<\tau$$
runs through the set of all 2-element subset of $\Sigma_E$ and
$$ \quad \bar{\Lambda}_A=\bar{\Lambda}_{\sigma}\wedge_{\bar{\BQ}} \bar{\Lambda}_{\tau}\cong\bar{\Lambda}_{\sigma}\otimes_{\bar{\BQ}} \bar{\Lambda}_{\tau}$$ is the corresponding one-dimensional $\bar{\mathcal{S}}_{E}^1$-invariant subspace;
 the action of $\bar{\mathcal{S}}_{E}^1$ on $\bar{\Lambda}_A$ is defined by the character $\delta_A^{1}$.

It follows from  \eqref{AneB} that if $W$ is a nonzero $\mathcal{S}_{E}^{1}$-invariant $\BQ$-vector subspace of $\Lambda_{\BQ}$ then $\bar{W}=W\otimes_{\BQ}\bar{\BQ}$ is a direct sum of some of $\bar{\Lambda}_A$.
The double transitivity condition implies that all $\bar{\Lambda}_A$'s are mutually 
Galois-conjugate over $\BQ$. It follows that  $\bar{W}=\wedge^2_{\bar{\BQ}}\bar{\Lambda}$, i.e.,
$W=\wedge^2_{\BQ}\Lambda_{\BQ}$ and we are done.

\end{proof}

\begin{rem}
See \cite{BZ} for explicit examples of complex tori that satisfy the conditions of Theorem \ref{Double2simple}.

\end{rem}

In the case of arbitrary simple complex tori (or even abelian varieties) the Hodge group may be neither semisimple nor commutative (see \cite{XuZ1,XuZ2,Xu} for explicit examples).  This is not the case for 2-simple tori in dimensions $>2$, in light of the following assertion.
\begin{prop}
\label{simpleCom}
Let $T$ be a $2$-simple torus of dimension $g >2$. (In particular, $T$ is simple.)
Then  $\mathrm{Hdg}(T)$ is either semisimple  or commutative.
The latter case occurs if and only if $\End^0(T)$ is a  number field of degree $2g$.

\end{prop}

\begin{proof}
We know (thanks to Theorem \ref{main}) that $E=\End^0(T)$ is a number field of degree 
$$[E:\BQ] \in \{1, g, 2g\}.$$

If $[E:\BQ]=1$ then $E=\BQ$. In light of Lemma \ref{center},  $\mathrm{Hdg}(T)$ is  semisimple.

If $[E:\BQ]=2g$ then $\Lambda_{\BQ}=\mathrm{H}_1(T,\BQ)$ is a one-dimensional $E$-vector space, i.e.,
$$E=\End_E\left(\mathrm{H}_1(T,\BQ)\right).$$
This implies that
$$\mathrm{hdg}_{T} \subset \End_E\left(\mathrm{H}_1(T,\BQ)\right)\subset E$$
and therefore $\mathrm{hdg}_{T}$ is a commutative $\BQ$-Lie algebra. It follows that $\mathrm{Hdg}(T)$ is commutative.

Assume that $[E:\BQ]=g$, i.e.,  $\Lambda_{\BQ}=\mathrm{H}_1(T,\BQ)$ is a two-dimensional $E$-vector space. Then
$$\mathrm{hdg}_{T}\subset \End_E\left(\mathrm{H}_1(T,\BQ)\right)\supset E \supset
\mathfrak{c}_T.$$
Let 
\begin{equation}
\label{TrE}
\mathrm{Tr}_E:  \End_E\left(\mathrm{H}_1(T,\BQ)\right) \to E
\end{equation}
be the (surjective) $E$-linear trace map, which is a homomorphism of $\BQ$-Lie algebras
(here we view $E$ as a commutative $\BQ$-Lie algebra); in addition, the restriction of $\mathrm{Tr}_E$ to $E$ is multiplication by
$$\dim_E\left(\mathrm{H}_1(T,\BQ)\right) =2.$$
We write $\sll\left(\mathrm{H}_1(T,\BQ)/E\right)$ for  $\ker(\mathrm{Tr}_E)$, which is an absolutely simple $E$-Lie algebra of traceless $E$-linear operators in $\mathrm{H}_1(T,\BQ)$.
(Viewed as the $\BQ$-Lie algebra, $\sll\left(\mathrm{H}_1(T,\BQ)/E\right)$  is  simple but not absolutely simple.)

On the other hand,  let
$${\det}_E: \Aut_E\left(\mathrm{H}_1(T,\BQ)\right) \to E^{*}$$
be the multiplicative {\sl determinant} map. Clearly,
$$\mathrm{Hdg}(T)(\BQ) \subset \Aut_E\left(\mathrm{H}_1(T,\BQ)\right)$$
and the group  $\Aut_E\left(\mathrm{H}_1(T,\BQ)\right)$ acts naturally on the one-dimensional $E$-vector space
$\Hom_E\left(\wedge^2_E \mathrm{H}_1(T,\BQ),E\right)$
 via the character $\det_E$.
I claim that ${\det}_E$ kills $\mathrm{Hdg}(T)(\BQ)$. Indeed, if this is not the case, then
the rational Hodge substructure $H_E \cong \Hom_E\left(\wedge^2_E \mathrm{H}_1(T,\BQ),E\right)$  
of $\mathrm{H}^2(T,\BQ)$ 
(defined in Section \ref{proofMain}, see \eqref{psiE} and \eqref{HE2})
has {\sl nonzero} $(2,0)$-component, whose $\BC$-dimension
$$\le \frac{\dim_{\BQ}(H_E)}{2}=\frac{[E:\BQ]}{2} =\frac{g}{2}<g,$$
which contradicts the $2$-simplicity of $T$.  Hence, ${\det}_E$ kills $\mathrm{Hdg}(T)(\BQ)$.
Taking into account that $\mathrm{Hdg}(T)(\BQ)$ is dense in $\Hdg(T)$ in Zariski topology and the minimality property in the definition of the Hodge group, we conclude that
$$\mathrm{Hdg}(T)\subset \mathrm{Res}_{E/\BQ}\SL(\left(\mathrm{H}_1(T,\BQ)/E\right)$$
where $\SL(\left(\mathrm{H}_1(T,\BQ)/E\right)$ is the special linear group of the $E$-vector space $\mathrm{H}_1(T,\BQ)$, which is a {\sl simple algebraic $E$-group}, and
$\mathrm{Res}_{E/\BQ}$ is the {\sl Weil restriction} of scalars. Taking into account that
the $\BQ$-Lie algebra $\sll\left(\mathrm{H}_1(T,\BQ)/E\right)$ is the Lie algebra of the $\BQ$-algebraic group
$\mathrm{Res}_{E/\BQ}\SL\left(\mathrm{H}_1(T,\BQ)/E\right)$, we conclude that
\begin{equation}
\label{slE}
\mathrm{hdg}_{T} \subset \sll\left(\mathrm{H}_1(T,\BQ)/E\right) \cong
\sll(2,E).
\end{equation}
In particular, $\mathrm{Tr}_E$ kills $\mathfrak{c}_T$. Since $\mathfrak{c}_T\subset E$,
$$0=\mathrm{Tr}_E(c)=2c \ \forall c \in \mathfrak{c}_T\subset E.$$
This implies that $ \mathfrak{c}_T=\{0\}$, i.e., $\mathrm{hdg}_{T}$ is semisimple, i.e., $\mathrm{Hdg}(T)$ is semisimple. This ends the proof.
\end{proof}

The following assertion may be viewed as a variant of a theorem of P. Deligne \cite{Deligne}
about abelian varieties (see also    \cite{SerreRennes} and \cite{Z84,Z86}).

\begin{thm}
\label{deligneT}
Let $T$ be a simple complex torus.
Let $\mathrm{hdg}_T \subset \End_{\BQ}(\Lambda_{\BQ})$ be the (reductive) $\BQ$-Lie algebra of $\mathrm{Hdg}(T)$,
whose natural representation in $\Lambda_{\BQ}$ is irreducible.

Let $\mathfrak{g}$ be a simple (non-abelian) factor of the complex reductive Lie algebra
$$\mathrm{hdg}_{T,\BC}=\mathrm{hdg}_T\otimes_{\BQ}\BC \subset \End_{\BQ}(\Lambda_{\BQ})\otimes_{\BQ}\BC
=\End_{\BC}(\Lambda_{\BC})=\End_{\BC}\left(V_{\BC}\right).$$

Then

\begin{itemize}
\item[(i)]
The simple complex Lie algebra $\mathfrak{g}$ is of classical type (${\sf{A}}_l, {\sf{B}}_l,{\sf{C}}_l,{\sf{D}}_l$) of a certain positive rank $l$.

\item[(ii)]

Let $W$ be a nontrivial simple $\mathfrak{g}$-submodule of $V_{\BC}$.  Then its highest weight
is minuscule one.  
\end{itemize}
\end{thm}

\begin{proof}
The result follows readily from 
Corollary \ref{JtoF} combined with Proposition 2.4.1 of \cite{Z84} applied to
$$k=\BC, k_0=\BQ, \ W=\Lambda_{\BQ}, \ \mathfrak{g}=\mathrm{hdg}_T, \ f =f_T, \ A=\{1,-1\}$$
and 
$$n=1, \ a_0=1, \ a_1=-1.$$
\end{proof}

The following assertion may be viewed as a variant of a theorem   of M.V. Borovoi 
about abelian varieties \cite{Bor}, see also \cite{Z91}.
\begin{thm}
\label{BorT}
Suppose that $T$ is a simple complex torus with $\End^0(T)=\BQ$.

Then its Hodge group $\mathrm{Hdg}(T)$ is a $\BQ$-simple linear algebraic group,
i.e., its $\BQ$-Lie algebra $\mathrm{hdg}_T$ is simple.
\end{thm}

\begin{proof}

Clearly, $\mathrm{hdg}_T$ is a semisimple $\BQ$-Lie algebra, whose natural faithful representation in $\Lambda_{\BQ}$
is {\sl absolutely irreducible}. By Remark \ref{HodgeElement},  
$f_T\in \mathrm{hdg}_{T,\BC}$ is a Hodge element of $\mathrm{hdg}_T$. The spectrum of the linear
semisimple operator $f_T$ in $\Lambda_{\BQ}$ consists of precisely two eigenvalues, $1$ and $-1$.
Now it follows from Theorem 1.5 of \cite{Z84} that $\mathrm{hdg}_T$ is simple.  This means that
$\mathrm{Hdg}(T)$ is a $\BQ$-simple algebraic group.

\end{proof}

\begin{cor}
\label{BorTabs}
Suppose that $T$ is a simple complex torus of dimension $g$ with $\End^0(T)=\BQ$.  Assume also
that $2g$ is not a power (e.g., $g$ is odd).

Then $\mathrm{Hdg}(T)$ is an absolutely simple $\BQ$-algebraic group that enjoys precisely one of the
following two properties.

\begin{itemize}
\item
$\mathrm{Hdg}(T)$ is 
of type ${\sf{A}}_{2g-1}, {\sf{C}}_g,{\sf{D}}_g$.

\item
$\mathrm{Hdg}(T)$ is 
of type  ${\sf{A}}_r$ where $r$ is a
 positive integer that enjoys the following properties.

$1<r<2g-1$ and there is an integer $j$ such that $1<j<2g-1$ and $2g=\binom{r+1}{j}$.
\end{itemize}
\end{cor}

\begin{proof}
By Theorem \ref{BorT}, $\mathrm{hdg}_T$ is a simple $\BQ$-Lie algebra.  Suppose that $\mathrm{hdg}_T$ is {\sl not} absolutely simple,
i.e., the complex Lie algebra $\mathrm{hdg}_{T,\BC}$ splits into a direct sum
$$\mathrm{hdg}_{T,\BC}=\oplus_{\ell=1}^d \mathfrak{g}_{\ell}$$
of $d$ simple complex Lie algebras $\mathfrak{g}_{\ell}$ where $d>1$.  We are going to prove that $2g$ is a $d$th power, which gives us a desired contradiction. 

The simplicity of $\mathrm{hdg}_T$ means that $\Aut(\BC)$ permutes the set $\{\mathfrak{g}_{\ell}\}_{\ell=1}^d$
{\sl transitively}. Namely, each $s \in \Aut(\BC)$ gives rise to the semi-linear automorphism of the $\BC$-vector space
$$ \Lambda_{\BC} \to \Lambda_{\BC},  \quad x\otimes z \mapsto x\otimes s(z) \ \forall x \in \Lambda_{\BQ}, z\in \BC$$
and to the semi-linear automorphism of the $\BC$-algebra
$$\End_{\BC}(\Lambda_{\BC}) \to \End_{\BC}(\Lambda_{\BC}),  u\otimes z \mapsto u \otimes s(z) \  \forall x \in \End_{\BQ}(\Lambda_{\BQ}), z\in \BC.$$
We continue to denote those automorphisms by $s$.

The transitivity means
that for each $\mathfrak{g}_{\ell}$ there is $s_{\ell}\in \Aut(\BC)$ such that
$$\mathfrak{g}_{\ell}=s_{\ell}\left(\mathfrak{g}_{1}\right).$$
We know that the $\mathrm{hdg}_{T,\BC}$-module $\Lambda_{\BC}$ is (absolutely) simple. Since each $\mathfrak{g}_{\ell}$
is a direct summand of  $\mathrm{hdg}_{T,\BC}$, the $\mathfrak{g}_{\ell}$-module $\Lambda_{\BC}$ is isotypic, i.e., there is
a simple $\mathfrak{g}_{\ell}$-submodule $W_{\ell}\subset \Lambda_{\BC}$ such that all simple $\mathfrak{g}_{\ell}$-submodules of $\Lambda_{\BC}$
are isomorphic to $W_{\ell}$.  In addition, the $\mathrm{hdg}_{T,\BC}\left(=\oplus_{\ell=1}^d \mathfrak{g}_{\ell}\right)$-module $\Lambda_{\BC}$ splits into a
tensor product $\otimes_{\ell=1}^d W_{\ell}$ of simple $\mathfrak{g}_{\ell}$-modules $W_{\ell}$. Let us prove that $\dim_{\BC}(W_{\ell})$ does {\sl not} depend on $\ell$.

Indeed,  $s_{\ell}(W_1)$ is a simple $\mathfrak{g}_{\ell}$-submodule of $\Lambda_{\BC}$ and therefore is isomorphic to $W_1$. This implies that
$\dim(W_1)=\dim(W_{\ell})$ and therefore 
$$2g=\dim_{\BC}(\Lambda_{\BC})=\left(\dim_{\BC}(W_1)\right)^d,$$
which gives us a desired contradiction. So,  $\mathrm{hdg}_{T,\BC}$ is a {\sl simple} complex Lie algebra and $\Lambda_{\BC}$ is a faithful simple $\mathrm{hdg}_{T,\BC}$-module.
 By Theorem \ref{deligneT},  $\mathrm{hdg}_{T,\BC}$  is a classical Lie algebra (of type
$\sf{A}_r, \sf{B}_r,\sf{C}_r$ or $\sf{D}_r$), and the highest weight of   $\Lambda_{\BC}$  is  {\sl minuscule}.
The remaining assertion follows readily from the inspection of dimensions of minuscule representations of classical Lie algebras
\cite[Tables]{Bourbaki78}.

\end{proof}

\begin{ex}
\label{classicalH}
Suppose that $T$ is a complex torus of dimension $g$ such that one of the following conditions holds. \begin{itemize}
\item[(i)]
$\mathrm{Hdg}(T)=\SL(\Lambda_{\BQ})$.

\item[(ii)] There exists a nondegenerate quadratic form 
$$\phi: \Lambda_{\BQ} \to \BQ$$
of even signature $(2p,2q)$ with $p+q=g\ge 3$
such that $\mathrm{Hdg}(T)$ coincides with the corresponding special orthogonal group.
$\mathrm{SO}(\Lambda_{\BQ},\phi)$.
\item[(iii)] There exists a nondegenerate alternating $\BQ$-bilinear form
$$\Lambda_{\BQ}\times \Lambda_{\BQ} \to \BQ$$
such that $\mathrm{Hdg}(T)$ coincides with the corresponding symplectic group
$\mathrm{Sp}(\Lambda_{\BQ})$.
\end{itemize}

Then $T$ is $2$-simple.

Indeed, in the cases (i) and (ii), (in the obvious notation) the natural representation of $\SL(\Lambda_{\BC})$ (resp. $\mathrm{SO}(\Lambda_{\BC})$) in $\wedge_{\BC}^2(\Lambda_{\BC})$ is  irreducible, see \cite[Ch. 8, Sect. 13]{Bourbaki78}. This implies that the natural representation of $\SL(\Lambda_{\BQ})$ (resp. $\mathrm{SO}(\Lambda_{\BQ},\phi)$) in $\wedge_{\BQ}^2(\Lambda_{\BQ})$ is absolutely irreducible. By duality, the same is true for $\Hom_{\BQ}(\wedge_{\BQ}^2(\Lambda_{\BQ}),\BQ)=\mathrm{H}^2(T,\BQ)$, i.e., the $\mathrm{Hdg}(T)$-module $\mathrm{H}^2(T,\BQ)$ is simple. This implies that  $T$ is $2$-simple. Notice that in these cases we deal with simple complex tori that are {\sl not} abelian varieties (since they do not carry nonzero 2-dimensional Hodge classes), and whose endomorphism algebra is $\BQ$.

In the case (iii), the natural representation of $\Sp(\Lambda_{\BC})$  in  $\wedge_{\BC}^2(\Lambda_{\BC})$ is  a direct sum of an irreducible representation and a trivial one-dimension representation \cite[Tables]{VO}. This implies that the natural representation of  $\Sp(\Lambda_{\BQ})$  in $\wedge_{\BQ}^2(\Lambda_{\BQ})$ is  a direct sum of an absolutely irreducible representation and a trivial one-dimension representation. By duality, the same is true for $\Hom_{\BQ}(\wedge_{\BQ}^2(\Lambda_{\BQ}),\BQ)=\mathrm{H}^2(T,\BQ)$, i.e., the $\mathrm{Hdg}(T)$-module $\mathrm{H}^2(T,\BQ)$ is a direct sum of an absolutely  simple module and a trivial module of $\BQ$-dimension $1$ The latter consists of all $\Hdg(T)$-invariants in $\mathrm{H}^2(T,\BQ)$, i.e., coincides with
$\mathrm{H}^{1,1}(T,\BQ)$. The former is an irreducible rational Hodge structure.  It follows from Remark \ref{translation}  (ii) that $T$ is 2-simple.\footnote{For abelian varieties $T$ the case (iii)  was done in \cite[Sect. 5.1]{AC}.} See \cite{ZarhinPrym} for explicit examples (in all dimensions) of complex abelian varieties $T$ with $\mathrm{Hdg}(T)=\mathrm{Sp}(\Lambda_{\BQ})$.

\end{ex}

\begin{thm}
\label{existH}
Let $\Pi_{\BQ}$ be a $\BQ$-vector space of positive even dimension $2g$, and $\mathcal{G}$ a $\BQ$-simple algebraic subgroup of $\GL(\Pi_{\BQ})$,
whose $\BQ$-Lie algebra $\mathfrak{g}$ may be viewed as a simple $\BQ$-Lie subalgebra of $\End_{\BQ}(\Pi_{\BQ})$. Let us consider the real Lie subalgebra
$$\mathfrak{g}_{\BR}=\mathfrak{g}\otimes_{\BQ}\BR \subset \End_{\BQ}(\Pi_{\BQ})\otimes_{\BQ}\BR=\End_{\BR}(\Pi_{\BR})$$
where $\Pi_{\BR}=\Pi_{\BQ}\otimes_{\BQ}\BR$ is the corresponding $2g$-dimensional real vector space. Suppose that there exists
an element
$$J_0 \in \mathfrak{g}_{\BR}\subset \End_{\BR}(\Pi_{\BR})$$
such that $J_0^2=-1$ in $ \End_{\BR}(\Pi_{\BR})$. 
Then there exists $J \in  \mathfrak{g}_{\BR}$ that enjoys the following properties.
\begin{itemize}
\item[(i)]
 $J^2=-1$.
 \item[(ii)]
Let us endow $\Pi_{\BR}$ with the structure of a $g$-dimensional complex vector space by defining
$$(a+b\mathbf{i}) v:=av+b J(v) \quad \forall a+b\mathbf{i} \in \BC \text{ with }a,b \in\BR.$$
Then for every discrete subgroup $\Lambda$ of rank $2g$ in $\Pi_{\BQ}$ the corresponding complex torus $T=\Pi_{\BR}/\Lambda$
has Hodge group $\mathcal{G}$.
\end{itemize}
\end{thm}

\begin{proof}
We will need the following auxiliary statement.
\begin{lem}
\label{idealW}
If $\mathfrak{v}$ is a proper nonzero $\BQ$-vector subspace of  $\mathfrak{g}$ then the real vector subspace
$$\mathfrak{v}_{\BR}=\mathfrak{v}\otimes_{\BQ}\BR\subset \mathfrak{g}\otimes_{\BQ}\BR=\mathfrak{g}_{\BR}$$
does not contain a nonzero ideal of $\mathfrak{g}_{\BR}$.
\end{lem}
Notice that  
the set of proper nonzero $\BQ$-Lie subalgebras $L$ of  $\mathfrak{g}$ is {\sl countable}. By Lemma \ref{idealW},
every $L_{\BR}=L\otimes_{\BQ}\BR$ does not contain a nonzero ideal of $\mathfrak{g}_{\BR}$. By \cite[Lemma 2 on p. 494]{ZarhinSh},
the closed subset
$$\mathcal{G}(L_{\BR},J_0)=\{u \in \mathcal{G}(\BR)\mid J_0\in u^{-1} L_{\BR} u\}$$
is nowhere dense in $\mathcal{G}(\BR)$. It follows that there exists $u \in  \mathcal{G}(\BR)$ such that
$J_0$ does not lie in any of $ u^{-1} L_{\BR} u$. Let us put
$$J:=u J_0 u^{-1} \in \mathfrak{g}_{\BR}\subset \End_{\BR}(\Pi_{\BR}).$$
Then $J$ does not lie in any of  $L_{\BR}$ and $J^2=-1 \in \End_{\BR}(\Pi_{\BR})$.
It follows that $\mathfrak{g}$ coincides with the smallest $\BQ$-Lie subalgebra $\mathfrak{u}$ of 
$\mathfrak{g}$  such that $\mathfrak{u}_{\BR}$ contains $J$. This implies that
$\mathfrak{g}$ coincides with the smallest $\BQ$-Lie subalgebra $\mathfrak{u}$ of 
$\End_{\BQ}(\Pi_{\BQ})$  such that $\mathfrak{u}_{\BR}$ contains $J$.  It follows readily that
$\mathfrak{g}$ coincides with the Lie algebra of the Hodge group of a complex torus $T=\Pi_{\BR}/\Lambda$
where the complex structure on the real vector space $\Pi_{\BR}$ is defined by $J$ and $\Lambda$ is any 
discrete  subgroup $\Lambda$ of rank $2g$ in $\Pi_{\BQ}$.
\end{proof}

\begin{proof}[Proof of Lemma \ref{idealW}]
Suppose that $L_{\BR}$ contains a nonzero ideal $\mathfrak{a}$ of $\mathfrak{g}_{\BR}$.
Then the $\BC$-vector subspace
$$L_{\BC}=L_{\BR}\otimes_{\BR}\BC=L\otimes_{\BQ}\BC$$
contains a nonzero ideal $\mathfrak{a}_{\BC}=\mathfrak{a}\otimes_{\BR}\BC$ of the complex Lie algebra
$$\mathfrak{g}_{\BC}=\mathfrak{g}_{\BR}\otimes_{\BR}\BC=\mathfrak{g}\otimes_{\BQ}\BC.$$
Then 
$$\tilde{\frak{a}}=\sum_{s \in \Aut(\BC)} s(\frak{a})$$
is a $\Aut(\BC)$-invariant ideal of $\mathfrak{g}_{\BC}$ that lies in $L_{\BC}$. Hence, there is a $\BQ$-vector subspace $\mathfrak{a}_{\BQ}$ such that
$$\tilde{\frak{a}}=\mathfrak{a}_{\BQ}\otimes_{\BQ}\BC;$$
in addition, $\mathfrak{a}_{\BQ}$ is an ideal of $\mathfrak{g}$, which contradicts the simplicity of the $\BQ$-Lie algebra $\mathfrak{g}$.
This ends the proof.
\end{proof}

\begin{ex}
\label{classicalHe}
We keep the notation of Theorem \ref{existH}.
Let $g \ge 3$ be an integer, $\Pi_{\BQ}$ a $2g$-dimensional vector space over $\BQ$.  Let
 $\mathcal{G}$ be a  $\BQ$-simple algebraic subgroup of $\GL(\Pi_{\BQ})$
that enjoys one of the following properties.

\begin{itemize}
\item[(i)]
 $\mathcal{G}=\SL(\Pi_{\BQ})$.
\item[(ii)] There exists a nondegenerate quadratic form 
$$\phi: \Pi_{\BQ} \to \BQ$$
of even signature $(2p,2q)$ with $p+q=g$
such that $\mathcal{G}$ coincides with the corresponding special orthogonal group
$\mathrm{SO}(\Pi_{\BQ},\phi)$.
\end{itemize}
In both cases there exists $J_0 \in \mathfrak{g}_{\BR}$ with $J_0^2=-1$.
(Here $\mathfrak{g}\subset \End_{\BQ}(\Pi_{\BQ})$ is the Lie algebra of $\mathcal{g}$.)
In light of Theorem \ref{existH}, there exists a complex structure on the real vector space $\Pi_{\BR}$ such that
the Hodge group of the corresponding complex torus $T=\Pi_{\BR}/\Lambda$ coincides with $\mathcal{G}$.
In light of Example \ref{classicalH}, $T$ is $2$-simple.
\end{ex}

\section{The degree $g$ case}
\label{DegreeG}

In this section we discuss $g$-dimensional $2$-simple tori, whose endomorphism algebra is a number field of degree $g$.

\begin{thm}
\label{degreeg}
Let $T$ be a $2$-simple torus of dimension $g>2$. If $\End^0(T)$ is a number field $E$ of degree $g$
then
$$\mathrm{Hdg}(T)=\mathrm{Res}_{E/\BQ}\SL(\left(\mathrm{H}_1(T,\BQ)/E\right).$$
\end{thm}

\begin{proof}
It suffices to check that 
\begin{equation}
\label{hdgSL}
\mathrm{hdg}_{T} = \sll\left(\mathrm{H}_1(T,\BQ)/E\right).
\end{equation}
In light of \eqref{slE},  the desired equality \eqref{hdgSL}
 is an immediate corollary of  the following  observation
 applied to 
 $$k=\BQ, \ K=E, W=\mathrm{H}_1(T,\BQ), \ \mathfrak{g}=\mathrm{hdg}_T.$$

\begin{lem}
\label{keyg}
Let $g$ be a positive integer, $W$ a $2g$-dimensional vector space over a field $k$ of characteristic $0$,  $\mathfrak{g} \subset \End_k(W)$ a linear  semisimple $k$-subalgebra such that the centralizer
$$K:=\End_{\mathfrak{g}}W\subset \End_k(W)$$
is an overfield of $K$ with $[K:k]=g$.  
Then $\mathfrak{g}$ coincides with the Lie algebra $\sll(W/K)$ of traceless $K$-linear operators in $W$. 
\end{lem}

\begin{proof}[Proof of Lemma \ref{keyg}]
The semisimplicity of $\mathfrak{g}$ implies that
\begin{equation}
\label{slEE}
\mathfrak{g}\subset \sll(W/K).
\end{equation}
In what follows we mimick the arguments of \cite[pp. 790--791, Proof of Th. 4.4.10]{Ribet} where $\ell$-adic Lie algebras are treated.

Let $\bar{k}$ be an algebraic closure of $k$, and $\Sigma_K$ the $g$-element set of field embeddings $\sigma: K \hookrightarrow \bar{k}$
that coincide with the identity map on $k$.
Let us consider the $2g$-dimensional $\bar{k}$-vector space $\bar{W}=W\otimes_k\bar{k}$ and the  $\bar{k}$-Lie algebra
\begin{equation}
\label{overBar}
\bar{\mathfrak{g}}=\mathfrak{g}\otimes_k \bar{k} \subset \End_k(W)\otimes_k \bar{k} =\End_{\bar{k}}(\bar{W}).
\end{equation}
The semisimplicity of the $k$-Lie algebra $\mathfrak{g}$ implies the {\sl semisimplicity} of 
the $\bar{k}$-Lie algebra $\bar{\mathfrak{g}}$.

Clearly, the {\sl centralizer}  $\End_{\bar{\mathfrak{g}}}(\bar{W})$
of  $\bar{\mathfrak{g}}$ in $\End_{\bar{k}}(\bar{W})$ equals
\begin{equation}
\label{central}
\End_{\mathfrak{g}}(W)\otimes_k\bar{k}=K\otimes_k\bar{k}
\end{equation}
and  $\bar{W}$ is a free $K\otimes_k\bar{k}$-module of rank $2$, because $W$ is a vector space over $K$ of dimension $2$.  We have
\begin{equation}
\label{EkK}
K\otimes_k\bar{k}=\oplus_{\sigma\in \Sigma_K}K\otimes_{K,\sigma}\bar{k}=\oplus_{\sigma\in \Sigma_K} \bar{k}_{\sigma}
\end{equation}
where
$$\bar{k}_{\sigma}=K\otimes_{K,\sigma}\bar{k}=\bar{k}.$$
We have
$$\bar{W}=\oplus_{\sigma\in \Sigma_K} \bar{W}_{\sigma} \ \text{ where } \bar{W}_{\sigma}=\bar{k}_{\sigma} \bar{W}\subset \bar{W}.$$
The freeness of the $K\otimes_k\bar{k}$-module $\bar{W}$ with rank $2$ implies that each $\bar{W}_{\sigma}$ is a $\bar{k}_{\sigma}$-vector space of dimension $2$.
Since 
$$\bar{k}_{\sigma}\subset K\otimes_k \bar{k}=\End_{\bar{\mathfrak{g}}}(\bar{W}),$$
each $\bar{W}_{\sigma}=\bar{k}_{\sigma} \bar{W}$ is a $\bar{\mathfrak{g}}$-invariant subspace of $\bar{W}$ and
the centralizer of $\bar{\mathfrak{g}}$ 
\begin{equation}
\label{centSigma}
\End_{\bar{\mathfrak{g}}}(\bar{W}_{\sigma})=\bar{k}_{\sigma}=\bar{k}.
\end{equation}

 Let 
 $$\bar{\mathfrak{g}}_{\sigma}\subset \End_{\bar{k}_{\sigma}}(\bar{W}_{\sigma})= \End_{\bar{k}}(\bar{W}_{\sigma})$$
 be the image of the natural $\bar{k}$-Lie algebra homomorphism 
 $$\bar{\mathfrak{g}} \to  \End_{\bar{k}}(\bar{W}_{\sigma}).$$
 The semisimplicity of $\bar{\mathfrak{g}}$ implies the semisimplicity of the Lie algebra $\bar{\mathfrak{g}}_{\sigma}$,
 because the latter is isomorphic to a quotient of the former. This implies that
 $$\bar{\mathfrak{g}}_{\sigma}\subset \sll(\bar{W}_{\sigma}) \cong \sll(2, \bar{k}_{\sigma})=\sll(2, \bar{k}).$$
 Taking into account \eqref{centSigma} and the semisimplicity of $\bar{\mathfrak{g}}_{\sigma}$, we conclude that
 \begin{equation}
 \label{sl2Sigma}
 \bar{\mathfrak{g}}_{\sigma}=\sll(\bar{W}_{\sigma})\cong  \sll(2, \bar{k}).
 \end{equation}
 This implies that
 \begin{equation}
 \label{directSigma}
 \bar{\mathfrak{g}}\subset \oplus_{\sigma\in \Sigma_K}\bar{\mathfrak{g}}_{\sigma}
 = \oplus_{\sigma\in \Sigma_K}\sll(\bar{W}_{\sigma})\subset
 \oplus_{\sigma\in \Sigma_K}\End_{\bar{k}}(\bar{W}_{\sigma}).
 \end{equation}
 Let $\sigma$ and $\tau$ be {\sl distinct} elements of $\Sigma_K$. Clearly, $\bar{W}_{\sigma}\oplus \bar{W}_{\tau}$ is a  $\bar{\mathfrak{g}}$-invariant
 subspace of $\bar{W}$.
 Let $\bar{\mathfrak{g}}_{\sigma,\tau}$ be the image
 of $\bar{\mathfrak{g}}$ in $\End_{\bar{k}}(\bar{W}_{\sigma}\oplus \bar{W}_{\tau})$. Since $\bar{\mathfrak{g}}_{\sigma,\tau}$ is isomorphic to a
 quotient of $\bar{\mathfrak{g}}$, it
  is a semisimple $\bar{k}$-Lie algebra such that
 $$\bar{\mathfrak{g}}_{\sigma,\tau}\subset \sll(\bar{W}_{\sigma})\oplus \sll(\bar{W}_{\tau})\subset 
 \End_{\bar{k}}(\bar{W}_{\sigma})\oplus \End_{\bar{k}}(\bar{W}_{\tau}) \subset \End_{\bar{k}}(\bar{W}_{\sigma}\oplus \bar{W}_{\tau}).$$
 Notice that $\bar{\mathfrak{g}}_{\sigma,\tau}$ projects {\sl surjectively} on both 
 $$\bar{\mathfrak{g}}_{\sigma}=\sll(\bar{W}_{\sigma}) \ \text{ and } \bar{\mathfrak{g}}_{\tau}=\sll(\bar{W}_{\tau}),$$
 because $\bar{\mathfrak{g}}$ does.
 The simplicity of both pairwise isomorphic Lie algebras $\sll(\bar{W}_{\sigma})$ and $\sll(\bar{W}_{\tau})$ and the semisimplicity
 of   $\bar{\mathfrak{g}}_{\sigma,\tau}$ implies that either
 \begin{equation}
 \label{gSigmaTau}
 \bar{\mathfrak{g}}_{\sigma,\tau}=\sll(\bar{W}_{\sigma})\oplus \sll(\bar{W}_{\tau})
 \end{equation}
 or
 $$\bar{\mathfrak{g}}_{\sigma,\tau}\cong \sll(\bar{W}_{\sigma})\cong  \sll(\bar{W}_{\tau})\cong \sll(2,\bar{k}).$$
 In the latter case the $\bar{\mathfrak{g}}_{\sigma,\tau}$-modules $\bar{W}_{\sigma}$ and $\bar{W}_{\tau}$ are isomorphic, because the Lie
 algebra $\sll(2,\bar{k})$ has precisely one nontrivial $2$-dimensional representation  over $\bar{k}$, up to an isomorphism. This implies that the
 $\bar{\mathfrak{g}}$-modules  $\bar{W}_{\sigma}$ and $\bar{W}_{\tau}$ are isomorphic  as well and therefore the centralizer
 $\End_{\bar{\mathfrak{g}}}(\bar{W})$ is noncommutative, which is not the case. The obtained contradiction proves that the equality \eqref{gSigmaTau}
 holds for any $\sigma,\tau$.  Now, it follows from Lemma on p. 790-791 of \cite{Ribet} that
 $$\bar{\mathfrak{g}}= \oplus_{\sigma\in \Sigma_K}\bar{\mathfrak{g}}_{\sigma}
 = \oplus_{\sigma\in \Sigma_K}\sll(\bar{W}_{\sigma}).$$
 This implies that
 $$\dim_{\bar{k}}(\bar{\mathfrak{g}})=3g=\dim_k\sll(W/K).$$
 By \eqref{slEE},  $\mathfrak{g}\subset \sll(W/K)$. Taking into account that
 $\dim_k(\mathfrak{g})=\dim_{\bar{k}}(\bar{\mathfrak{g}})$, we conclude that
 $\dim_k(\mathfrak{g})=\dim_k \sll(W/K).$  This implies that
 $\mathfrak{g}= \sll(W/K)$, which ends the proof.

\end{proof}

\end{proof}

\begin{thm}
\label{exteriorH}
Let $T$ a  simple complex torus of dimension $g>2$. Suppose that $\End^0(T)$ is a number field $E$ of degree $g$
and
$$\mathrm{Hdg}(T)=\mathrm{Res}_{E/\BQ}\SL(\left(\mathrm{H}_1(T,\BQ)/E\right).$$
Then the following conditions are equivalent.

\begin{itemize}
\item[(i)]
$T$ is $2$-simple.

\item[(ii)]
$E$ is almost doubly transitive.

\end{itemize}

\end{thm}

It follows from Remark \ref{translation}(ii) 
that Theorem \ref{exteriorH} is an immediate corollary of the following observation
 applied to 
$$k=\BQ, \  K=E, \ W=\mathrm{H}_1(T,\BQ).$$

\begin{lem}
\label{keyG}
Let $g$ be a positive integer $\ge 2$,  $W$ a $2g$-dimensional vector space over a field $k$ of characteristic $0$, and
 $\mathfrak{g} \subset \End_k(W)$ a linear  semisimple $k$-subalgebra such that the centralizer of $\mathfrak{g}$
$$K:=\End_{\mathfrak{g}}W\subset \End_k(W)$$
is an overfield of $k$ withbeta:= $[K:k]=g$, 
and $\mathfrak{g}=\sll(W/K)$  is the Lie algebra of traceless $K$-linear operators in $W$. 

Then:

\begin{itemize}
\item[(a1)]
The subspace 
$ \left(\wedge^2_k W\right)^{ \mathfrak{g}}$ of $\mathfrak{g}$-invariants 
does {\sl not} coincide with the whole  space $\wedge^2_k W$.
\item[(a2)]
The subspace 
$\Hom(\wedge^2_k W,k)^{ \mathfrak{g}}$ of $\mathfrak{g}$-invariants 
does {\sl not} coincide with the whole  space $\Hom(\wedge^2_k W,k)$.
\item[(b)]
 The following conditions are equivalent.
\begin{itemize}
\item[(b1)]
The $\mathfrak{g}$-module $\wedge^2_k W$ is a direct sum of its submodule 
$\left(\wedge^2_k W\right)^{ \mathfrak{g}}$ of $\mathfrak{g}$-invariants
and a simple $\mathfrak{g}$-module.
\item[(b2]
The $\mathfrak{g}$-module $\Hom(\wedge^2_k W,k)$ is a direct sum of its submodule 
$\Hom(\wedge^2_k W,k)^{ \mathfrak{g}}$ of $\mathfrak{g}$-invariants
and a simple $\mathfrak{g}$-module.  
\item[(b3)]
Let $\mathrm{Gal}(k)=\Aut(\bar{k}/k)$ be the absolute Galois group of $k$.  Let $\Sigma_K$ be the set
of $k$-linear field embeddings $K \hookrightarrow \bar{k}$.
Then the natural action
of $\mathrm{Gal}(k)$ on $\Sigma_K$ is almost doubly transitive.
\end{itemize}
\end{itemize}
\end{lem}

\begin{rem}
\label{equiv}
The equivalences of properties  (a1) and (a2) and the equivalence of of conditions (b1) and (b2)  follow readily from the semisimplicity of  $\mathfrak{g}$.

\end{rem}

\begin{cor}
\label{g3}
Let $T$ be a  simple complex torus of dimension $3$. Suppose that $\End^0(T)$ is a cubic number field $E$ 
and
$$\mathrm{Hdg}(T)=\mathrm{Res}_{E/\BQ}\SL(\left(\mathrm{H}_1(T,\BQ)/E\right).$$
Then 
$T$ is $2$-simple.
\end{cor}

\begin{proof}[Proof of Corollary \ref{g3}]
The result follows readily from Theorem \ref{exteriorH},
since every transitive action on the $3$-element set $\Sigma_E$ is almost doubly transitive.
\end{proof}

\begin{proof}[Proof of Lemma \ref{keyG}]
We use the notation of the Proof of Lemma \ref{keyg}.   In particular,
$$\bar{\mathfrak{g}}_{\sigma}=\sll(\bar{W}_{\sigma}), \ \bar{\mathfrak{g}}_{\sigma,\tau}=\sll(\bar{W}_{\sigma})\oplus \sll(\bar{W}_{\tau}) \
\forall \sigma, \tau \in \Sigma_K, \sigma \ne \tau;$$
$$\bar{W}=\oplus_{\sigma\in \Sigma_K}\bar{W}_{\sigma}, \ \bar{\mathfrak{g}}=\oplus_{\sigma\in \Sigma_K}\bar{\mathfrak{g}}_{\sigma}, \ .$$

Let us start with the $\mathfrak{g}$-module 
$$W^{\otimes 2}:=W\otimes_k W \to W.$$
There is an involution
$$\delta: W^{\otimes 2} \to W^{\otimes 2},  \ u\otimes v \mapsto v \otimes u,$$
whose subspace of invariants is the symmetric square $\mathbf{S}_k^2W$  of $W$ and the subspace of anti-invariants is the exterior square $\wedge_k^2W$.
Clearly, $\delta$ commutes with the action of $\mathfrak{g}$; in particular, both $\mathbf{S}_k^2W$ and  $\wedge_k^2W$ are $\mathfrak{g}$-invariant subspaces of the tensor square of $W$.

Let us consider 
the $\bar{\mathfrak{g}}$-module
$$\bar{W}^{\otimes 2}:=\bar{W}\otimes_{\bar{k}}\bar{W}.$$ 
Extending by $\bar{k}$-linearity  the involution $\delta$, we get the involution
$$\bar{\delta}: \bar{W}^{\otimes 2} \to \bar{W}^{\otimes 2},  \ u\otimes v \mapsto v \otimes u,$$
whose subspace of invariants is the symmetric square $\mathbf{S}_{\bar{k}}^2\bar{W}$  of $\bar{W}$ and 
the subspace of anti-invariants is the exterior square $\wedge_{\bar{k}}^2\bar{W}$.
Clearly, $\bar{\delta}$ commutes with the action of $\overline{\mathfrak{g}}$; in particular, both $\mathbf{S}_{\bar{k}}^2\bar{W}$  and  
 $\wedge_{\bar{k}}^2\bar{W}$  are $\overline{\mathfrak{g}}$-invariant subspaces of $\bar{W}^{\otimes 2}$.
 
 Let us choose an {\sl order} on $\Sigma_K$. Let $\Sigma_{K,2}$ be the set of all two-element subsets $B$ of $\Sigma_K$
 with 
 \begin{equation}
 \label{Bst}
 B=\{\sigma, \tau\};  \ \sigma,\tau \in \Sigma_K; \ \sigma<\tau.
 \end{equation}

Let us consider the following decomposition of
the $\bar{\mathfrak{g}}$-module  $\bar{W}^{\otimes 2}$ into a direct sum of $\bar{\delta}$-invariant $\bar{\mathfrak{g}}$-submodules
\begin{equation}
\label{splitting}
\bar{W}^{\otimes 2}=\Big(\oplus_{\sigma\in \Sigma_K}\bar{W}_{\sigma}\otimes_{\bar{k}}\bar{W}_{\sigma}\Big)\oplus
\Big(\oplus_{B=\{\sigma,\tau\}\in \Sigma_{K, 2}}(\bar{W}_{\sigma}\otimes_{\bar{k}}\bar{W}_{\tau})\oplus (\bar{W}_{\tau}\otimes_{\bar{k}}\bar{W}_{\sigma})\Big).
\end{equation}
Clearly, the action of the Lie algebra $\bar{\mathfrak{g}}$ on the tensor product $\bar{W}_{\sigma}\otimes_{\bar{k}}\bar{W}_{\sigma}$ factors through 
$$\bar{\mathfrak{g}}_{\sigma}= \sll(\bar{W}_{\sigma})$$
while the action of $\bar{\mathfrak{g}}$ on 
$$\bar{W}^{(B)}:=\left(\bar{W}_{\sigma}\otimes_{\bar{k}}\bar{W}_{\tau}\right)\oplus \left(\bar{W}_{\tau}\otimes_{\bar{k}}\bar{W}_{\sigma}\right)$$ factors
through 
$$\bar{\mathfrak{g}}_{(B)}:=\bar{\mathfrak{g}}_{\sigma,\tau}=\sll(\bar{W}_{\sigma})\oplus \sll(\bar{W}_{\tau}) \ \text{with } B=\{\sigma,\tau\}.$$
We have
$$\bar{W}_{\sigma}\otimes_{\bar{k}}\bar{W}_{\sigma}=\mathbf{S}_{\bar{k}}^2\bar{W}_{\sigma}\oplus \wedge_{\bar{k}}^2\bar{W}_{\sigma}$$
where first summand is a simple $\bar{\mathfrak{g}}_{\sigma}$-module that lies  in  $\mathbf{S}_{\bar{k}}^2\bar{W}$ while the action of $\bar{\mathfrak{g}}_{\sigma}$ (and therefore of $\bar{\mathfrak{g}}$)
on the second one is {\sl trivial}.

Both $\bar{\mathfrak{g}}_{(B)}=\bar{\mathfrak{g}}_{\sigma,\tau}$-modules $\bar{W}_{\sigma}\otimes_{\bar{k}}\bar{W}_{\tau}$ and  $\bar{W}_{\tau}\otimes_{\bar{k}}\bar{W}_{\sigma}$
are faithful simple; in addition, they are  isomorphic. Let us split $\bar{W}^{(B)}$ into a direct sum
$$\bar{W}^{(B)}=\bar{W}_{+}^{(B)}\oplus \bar{W}_{-}^{(B)}$$
of the subspaces $\bar{W}_{+}^{(B)}$ of $\bar{\delta}$-invariants  and $\bar{W}_{-}^{(B)}$ of $\bar{\delta}$-antiinvariants. Clearly, both subspaces
are nonzero  $\bar{\mathfrak{g}}_{\sigma,\tau}$-invariant subspaces and therefore are {\sl nontrivial simple} $\bar{\mathfrak{g}}_{\sigma,\tau}$-modules  that are isomorphic to
$$\bar{W}_{\sigma}\otimes_{\bar{k}}\bar{W}_{\tau} \cong \bar{W}_{\tau}\otimes_{\bar{k}}\bar{W}_{\sigma}.$$
The last sentence remains true if we replace ``$\bar{\mathfrak{g}}_{\sigma,\tau}$-modules'' by ``$\bar{\mathfrak{g}}$-modules''.
Obviously,
$$\bar{W}_{+}^{(B)}\subset \mathbf{S}_{\bar{k}}^2\bar{W}, \ \bar{W}_{-}^{(B)}\subset \wedge_{\bar{k}}^2\bar{W}.$$
It follows from \eqref{splitting} that the 
$\bar{\mathfrak{g}}$-module 
$\wedge_{\bar{k}}^2\bar{W}$ splits into a direct sum of the trivial 
$\bar{\mathfrak{g}}$-module $\oplus_{\sigma\in \Sigma_K}\wedge_{\bar{k}}^2\bar{W}_{\sigma}$ and a direct sum $\oplus_{B\in \Sigma_{K,2}} \bar{W}_{-}^{(B)}$ of nontrivial
mutually non-isomorphic simple $\bar{\mathfrak{g}}$-modules 
$\bar{W}_{-}^{(B)}$. So,
\begin{equation}
\label{S2}
\mathbf{S}_{\bar{k}}^2\bar{W}=\left(\oplus_{\sigma\in \Sigma_K} \mathbf{S}_{\bar{k}}^2\bar{W}_{\sigma}\right)\oplus 
\left(\oplus_{B\in \Sigma_{K,2}}\bar{W}_{+}^{(B)}\right);
\end{equation}
\begin{equation}
\label{Wedge2}
\wedge_{\bar{k}}^2\bar{W}=\left(\oplus_{\sigma\in \Sigma_K} \wedge_{\bar{k}}^2\bar{W}_{\sigma}\right)\oplus 
\left(\oplus_{B\in \Sigma_{K,2}}\bar{W}_{-}^{(B)}\right).
\end{equation}

Clearly,  
$$\bar{W}_{-}^{\{0\}}:=\oplus_{\sigma\in \Sigma_K} \wedge_{\bar{k}}^2\bar{W}_{\sigma}$$
coincides with the subspace of all $\bar{\mathfrak{g}}$-invariants in $\wedge_{\bar{k}}^2\bar{W}$.

Notice that   $\Gal(k)$ acts  naturally on both $\Sigma_K$ and $\Sigma_{K,2}$ in such a way that for all $s \in \Gal(k)$
\begin{equation}
\label{GaloisInv}
s\left(\wedge_{\bar{k}}^2\bar{W}_{\sigma}\right)= \wedge_{\bar{k}}^2\bar{W}_{s\sigma}, \ s\left(\bar{W}_{-}^{(B)}\right)=\bar{W}_{-}^{(sB)}.
\end{equation}
 It follows  that
$$s \bar{W}_{-}^{\{0\}}=\bar{W}_{-}^{\{0\}}.$$ It is also clear that if we put
\begin{equation}
\label{barU}
\bar{U}:=\oplus_{B\in \Sigma_{K,2}}\bar{W}_{-}^{(B)}\subset \wedge_k^2 W.
\end{equation}
then $s \bar{U}=\bar{U}$ for all $s \in \Gal(k)$.  This implies that
both $\bar{W}_{-}^{\{0\}}$  and $\bar{U}$ are {\sl nonzero} $\bar{k}$-vector subspaces that
 are defined over $k$, i.e.,
there are  {\sl nonzero} vector $k$-subspaces
$W_{-}^{\{0\}}$ and $U$ of $\wedge_k^2 W$ such that
$$\bar{W}_{-}^{\{0\}}=W_{-}^{\{0\}}\otimes_k \bar{k}, \ \bar{U}=U\otimes_k \bar{k}.$$
It follows from \eqref{Wedge2} and \eqref{barU} that
\begin{equation}
\label{directS}
\wedge_k^2 W=W_{-}^{\{0\}}\oplus U.
\end{equation}
The $\bar{\mathfrak{g}}$-invariance of both $\bar{k}$-vector subspaces $\bar{W}_{-}^{\{0\}}$ and $\bar{U}$ implies that both $k$-vector subspaces
$W_{-}^{\{0\}}$ and $U$ are $\mathfrak{g}$-submodules of $\wedge_k^2 W$. It is also clear that $W_{-}^{\{0\}}$ coincides with the subspace 
$\left(\wedge^2_k W\right)^{ \mathfrak{g}}$ of all
$\mathfrak{g}$-invariants in $\wedge_k^2 W$. 
Since $U \ne\{0\}$, it follows from \eqref{directS} that
\begin{equation}
\label{notWhole}
 \left(\wedge^2_k W\right)^{ \mathfrak{g}} \ne \wedge^2_k W,
 \end{equation}
which proves (a1) that, in turn, implies (a2), in light of Remark \ref{equiv}.

Combining the equality $\left(\wedge^2_k W\right)^{ \mathfrak{g}}=W_{-}^{\{0\}}$ with \eqref{directS}, we obtain that the property (b1) of our Lemma 
is equivalent to the simplicity of the $\mathfrak{g}$-module $U$.

Let $O$ be a $\Gal(k)$-{\sl orbit} in $\Sigma_{K,2}$. Let us consider the corresponding $\bar{\mathfrak{g}}$-{\sl submodule} of $\bar{U}$ defined by
\begin{equation}
\label{UO}
\bar{U}^{O}=\sum_{B\in O}\bar{W}_{-}^{(B)}.
\end{equation}
Clearly, $s \bar{U}^{O}=\bar{U}^{O}$ for all $s \in \Gal(k)$. This means that $\bar{U}^{O}$
is defined over $k$, i.e., there is a $\mathfrak{g}$-submodule $U^{O}$ of $U$ such that
$$\bar{U}^{O}=U^{O}\otimes_k \bar{k}.$$
 Since all the summands  in  the RHS of \eqref{UO} are pairwise non-isomorphic
simple $\bar{\mathfrak{g}}$-modules that (in light of \eqref{GaloisInv})  are permuted  transitively by $\Gal(k)$,
we conclude  that $U^{O}$ is a {\sl simple} $\mathfrak{g}$-submodule of $U$. Clearly, $U^{O}=U$ if and only if $O=\Sigma_{K,2}$,
i.e., if and only if the action of $\Gal(k)$ on $\Sigma_{K,2}$ is {\sl transitive}. This implies that the $\mathfrak{g}$-module $U$ is simple
if and only if  the action of $\Gal(k)$ on $\Sigma_{K,2}$ is {\sl transitive}. It follows that conditions (b1) and (b3) of our Lemma
are equivalent. By Remark \ref{equiv},  conditions  (b1) and (b2) are equivalent.  This ends the proof.

\end{proof}

\begin{thm}
\label{existenceg}
Let $E$ be a number field of degree $g>2$. Then there exists a simple $g$-dimensional complex torus $T=V/\Lambda$ such that
$$\End^0(T)=E, \ \mathrm{Hdg}(T)=\mathrm{Res}_{E/\BQ}\SL(\left(\mathrm{H}_1(T,\BQ)/E\right).$$
In particular, $T$ is $2$-simple if and only if $E$ is almost doubly transitive.

\end{thm}

\begin{proof} Let us consider the matrix
$$J_0=
\begin{bmatrix}
0 & 1\\
-1& 0
\end{bmatrix}
\in \Mat_2(\BQ) \subset \Mat_2(E)\subset \Mat_2(E_{\BR})$$
where $E_{\BR}:=E\otimes_{\BQ}\BR$ is the realification of $E$. By \cite[Prop. 2.8 on p. 19]{OZ},
there is $u \in \Mat_2(E_{\BR})$ such that
$$J=\exp(u) J_0 \exp(-u) =\exp(u) J_0 \left(\exp(u)\right)^{-1} \in  \Mat_2(E_{\BR})$$
enjoys the following property.  If $D$ is a $\BQ$-subalgebra of $\Mat_2(E)$ such that
$D_{\BR}=D\otimes_{\BQ}\BR$ contains $J$ then $D= \Mat_2(E)$.
Notice that 
$$J_0^2=-1, \quad J_0 \in \sll_2(E_{\BR})\subset  \Mat_2(E_{\BR})$$
It follows that
$$J^2=-1, \quad J \in \sll_2(E_{\BR})\subset \Mat_2(E_{\BR}).$$

Let $\mathfrak{g}$  be the smallest $\BQ$-Lie subalgebra of  $\Mat_2(E)$ such that its realification
$$\mathfrak{g}_{\BR}=\mathfrak{g}\otimes_{\BQ}\BR \subset \Mat_2(E)\otimes_{\BQ}\BR=\Mat_2(E_{\BR})$$
contains $J$. Clearly,
\begin{equation}
\label{centerE}
\mathfrak{g}\subset  \sll_2(E)
\end{equation}
 and the $\BQ$-subalgebra of $\Mat_2(E)$  generated by $\mathfrak{g}$
coincides with $\Mat_2(E)$. It makes the $2g$-dimensional  $\BQ$-vector space 
$$E^2=E\oplus E$$  a faithful simple  $\mathfrak{g}$-module
such that the centralizer of  $\mathfrak{g}$ in $\End_{\BQ}(E^2)$ coincides with $E$. This implies that $\mathfrak{g}$ is a reductive $\BQ$-Lie algebra
and its center lies in $E$. In light of \eqref{centerE}, this center is $\{0\}$, i.e.,  $\mathfrak{g}$ is a semisimple $\BQ$-Lie algebra. Applying
Lemma \ref{keyg} to 
$$k=\BQ, \ K=E, \  W=E^2,$$
 we conclude that
\begin{equation}
\label{slGequal}
\mathfrak{g}=  \sll_2(E).
\end{equation}
Now we are ready  to construct the desired complex torus $T$.  The operator $J$ provides the structure of a complex vector space on
$$V:=E^2\otimes_{\BQ}\BR=E_{\BR}^2=E_{\BR}\oplus E_{\BR}$$
such that $J \in \End_{\BR}(V)$ defines multiplication by $\mathbf{i}$. Pick any $\BZ$-lattice of rank $2g$ in $E^2$ and put
$T:=V/\Lambda$.  One may naturally identify $\Lambda\otimes \BQ$ with $E^2$.  In light of Theorem \ref{minLie}, the $\BQ$-Lie algebra $\mathrm{hdg}_T$ coincides
with $\mathfrak{g}$, i.e., $\mathrm{hdg}_T=\sll_2(E)$. It follows that $\mathrm{Hdg}(T)=\mathrm{Res}_{E/\BQ}\SL(\left(\mathrm{H}_1(T,\BQ)/E\right)$,
which ends the proof of the forst assertion of our Theorem. Now the second one follows from Theorem \ref{exteriorH}.

\end{proof}
 
\begin{proof}[Proof of Theorem \ref{existRealT}]
The first assertion follows readily from Theorem \ref{existenceg} combined with Theorem \ref{existReal} applied to $n=g$.

In order to prove the second assertion, one should take
$$\mathbf{s}=g-d-1\ge 0, \quad \mathbf{r}=g-2 \mathbf{s}=g-2(g-d-1)=2(d+1)-g\ge 0.$$

\end{proof}

\section{Semi-linear algebra}
\label{semiLin}
This section contains auxiliary results that will be used for the study of Hodge groups of $2$-simple tori without nontrivial endomorphisms.
In what follows $k$ stands for a field of characteristic $0$ and $K$ for an overfield of $k$ such that the automorphism group $\Aut(K/k)$ of $k$-linear
automorphisms of $K$ enjoys the following property.

{\sl The subfield $K^{\Aut(K/k)}$ of $\Aut(K/k)$-invariants coincides with $k$}. (This property holds if $K$ is an algebraically closed field, see Proposition \ref{AutKk} below.)

\begin{defn}
\label{Vsigma}
Let $\mathcal{V}$ be a finite-dimensional vector space over $K$ and $\sigma \in \Aut(K/k)$. Then the finite-dimensional vector space $^{\sigma}V$ over $K$ 
is defined as follows. Viewed as an additive group, $^{\sigma}V$  coincides with $V$ but multiplication by elements  $a \in K$ is defined in $^{\sigma}V$ by the formula
$$a, v \mapsto \sigma^{-1}(a)v.$$
\end{defn}
Clearly,
$$\dim_K(\mathcal{V})=\dim_K(^{\sigma}\mathcal{V}).$$
\begin{rem}
\label{sigmaM}
\begin{itemize}
\item[(1)]
If $x \in \End_K(\mathcal{V})$ is a $K$-linear operator in $\mathcal{V}$ then
$$x(\sigma^{-1}(a)v)=\sigma^{-1}(a)x(v) \quad \forall a \in K, v \in \mathcal{V}.$$
In other words, one may view $x$ as a $K$-linear operator in $^{\sigma}\mathcal{V}$ that we denote by 
${^{\sigma}\mathrm{id}}(x) \in  \End_K(^{\sigma}\mathcal{V})$. 
\item[(2)]
Let $m:=\dim_K(\mathcal{V})>0$, and $\{e_1, \dots, e_m\}$ be a basis of $\mathcal{V}$. Then one may view
$\{e_1, \dots, e_m\}$ as a basis of ${^{\sigma}\mathcal{V}}$.

If $A=(a_{ij})_{i,j=1}^m$ is the matrix of $x \in \End_K(\mathcal{V})$ with respect to $\{e_1, \dots, e_m\}$
then  obviously $\sigma(A)=(\sigma(a_{ij}))_{i,j=1}^m$ is the matrix of ${^{\sigma}x} \in \End_K(^{\sigma}\mathcal{V})$ with respect to $\{e_1, \dots, e_m\}$.
\end{itemize}
\end{rem}
\begin{lem}
\label{EndVsigma}
The formula
$$^{\sigma}\mathrm{id}: \End_K(\mathcal{V}) \to \End_K({^{\sigma}\mathcal{V}}),  \quad x \mapsto 
^{\sigma}\mathrm{id}(x)=
\{v \mapsto x(v)\} \ \forall x \in \End_K(\mathcal{V}),
 \ v \in {^{\sigma}\mathcal{V}}=\mathcal{V}$$
defines a ring isomorphism that
enjoys the following properties.

\begin{itemize}
\item[(i)]
${^{\sigma}\mathrm{id}}(ax) =\sigma(a)\cdot  {^{\sigma}\mathrm{id}}(x) \quad \forall a \in K, x \in \End_K(\mathcal{V}).$
\item[(ii)]
Let
$$\mathcal{P}_{x, \min}(t), \ \mathcal{P}_{x, \fchar}(t)\in K[t]$$
be the minimal and characteristic polynomials of $x$ respectively.  

Then  the minimal and characteristic polynomials of ${^{\sigma}\mathrm{id}}(x)$ coincide with
$\sigma(\mathcal{P}_{x, \min}(t))$ and $\sigma(\mathcal{P}_{x, \fchar}(t))$ respectively.
\item[(iii)]
If $a\in K$ is the trace of $x \in \End_K(\mathcal{V}) $ then $\sigma(a)$ is  the trace of 
${^{\sigma}\mathrm{id}}(x) \in  \End_K({^{\sigma}\mathcal{V}})$.
\end{itemize}
\end{lem}

\begin{proof}
(i)  is obvious.
Both assertions (ii) and (iii) follow from Remark \ref{sigmaM}.
\end{proof}

Let $\mathcal{V}_0$ be a finite-dimensional $k$-vector space and 
$$\mathcal{V}:=\mathbf{T}_{k,K}(\mathcal{V})=\mathcal{V}_0\otimes_k K$$ the corresponding $K$-vector space endowed by the following semi-linear action of $\Aut(K/k)$.
$$\sigma(v_0 \otimes a)=v_0 \otimes \sigma(a) \quad \forall a \in K, v_0 \in \mathcal{V}_0.$$
We will identify $\mathcal{V}_0$ with the $k$-vector subspace
$$\mathcal{V}_0 \otimes 1=\{v_0\otimes 1\mid v_0 \in \mathcal{V}_0\}\subset
\mathcal{V}_0\otimes_k K=\mathcal{V} .$$
Clearly, the $k$-vector subspace  $\mathcal{V}_0=\mathcal{V}_0 \otimes 1$ coincides with the $k$-vector subspace $\mathcal{V}^{\Aut(K/k)}$ of $\Aut(K/k)$-invariants.

The next asertion is probably known but I was unable to find a reference.
\begin{lem}
\label{subspace}
Let $\mathcal{W}$ be a $K$-vector subspace of $\mathcal{V}$. Then the following conditions are equivalent.
\begin{itemize}
\item[(i)]
$\mathcal{W}$ is $\Aut(K/k)$-invariant.
\item[(ii)]
There exists a $k$-vector subspace $\mathcal{W}_0$ of $\mathcal{V}_0$ such that
$$\mathcal{W}=\mathbf{T}_{k,K}(\mathcal{W}_0)=\mathcal{W}_0\otimes_k K
\subset\mathcal{V}_0\otimes_k K=\mathbf{T}_{k,K}(\mathcal{V}_0).$$
\end{itemize}
If this is the case then  $\mathcal{W}_0=\mathcal{W} \cap\mathcal{V}_0$.
\end{lem}

\begin{proof}
Let us put
$$m:=\dim_k(\mathcal{V}_0)=\dim_K(\mathcal{V});  \quad n:=\dim_K(\mathcal{W}) \le m.$$
If either $n=0$ or $n=m$  then the desired result is obvious. So, we may and will assume that
$0<n<m$, i.e.,
$$1 \le n \le m-1; \quad m \ge 2.$$

Let us fix a $k$-basis $\{e_1, \dots, e_m\}$ of $\mathcal{V}_0$, which we will also view as a $K$-basis of $\mathcal{V}$.

{\bf Step 1}. Assume that $n=1$. 
Take a nonzero vector $w \in \mathcal{W}$. Then  at least one of its coefficients with respect to our basis is not $0$, i.e.,
$$w=\sum_{i=1}^n a_i e_i,  \ a_i\in K$$
and $\exists j \in \{1, \dots,n\}$ such that $a_j \ne 0$.   Replacing $w$ by $a_j^{-1}w \in \mathcal{W}$, we may and will assume that $a_j=1$.
Then
$$K\cdot w=\mathcal{W}\ni \sigma(w)=\sum_{i=1}^n \sigma(a_i) e_i \quad \forall \sigma \in \Aut(K/k).$$
We have
$$\sigma(a_j)=\sigma(1)=1=a_j$$
and 
$$\sigma(w) \in \mathcal{W}=K \cdot w.$$
Since both $w$ and $\sigma(w)$ have the same (non-zero) $j$th coordinate,
we conclude that $\sigma(w)=w$ for all $\sigma$, i.e., all the coefficients $a_i \in k$ and therefore $w\in \mathcal{V}_0$ and $\mathcal{W}=\mathcal{W}_0\otimes_k K$ with
$$\mathcal{W}_0=k \cdot w \subset \mathcal{V}_0.$$
So, we have proven our assertion in the case of $n=1$. 

{\bf Step 2}. Let us prove that the $k$-vector subspace of $\Aut(K)$-invariants
\begin{equation}
\label{Winv}
\tilde{\mathcal{W}}_0:=\mathcal{W}^{\Aut(K/k)}=\mathcal{W}\cap V_0
\end{equation}
is {\sl not} $\{0\}$. Let us use induction by $n$ and $m$. By Step 1, our assertion is true for $n=1$. This implies its validity for for $m=2$. So, we may assume that
$$1<n<m, \quad m>2.$$
Let us consider the hyperplanes 
$$\mathcal{H}_0=\sum_{i=1}^{m-1}k\cdot  e_i\subset V_0, \quad \mathcal{H}=\mathbf{T}_{k,K}(\mathcal{H}_0)=\mathcal{H}_0\otimes_k K=
\sum_{i=1}^{m-1}K\cdot  e_i\subset \mathcal{V}.$$
Clearly, the intersection 
$$\mathcal{W}_H=\mathcal{W}\cap \mathcal{H}\subset \mathcal{H}$$
is an $\Aut(K/k)$-invariant subspace of both $\mathcal{W}$ and $\mathcal{H}$.
Clearly, either $\mathcal{W}_H=\mathcal{W}$ or $\dim_K(\mathcal{W}_H)=n-1>0$.
In the former case $\mathcal{W}_H$  is $\Aut(K/k)$-invariant subspace of the $(m-1)$-dimensional
$K$-vector space $\mathcal{H}=\mathcal{H}_0\otimes_k K$. Now the induction assumption for $m$ (applied to $\mathcal{H}$  instead of $\mathcal{V}$) implies that
$$\tilde{\mathcal{W}}_0=\mathcal{W}^{\Aut(K/k)}=(\mathcal{W}_H)^{\Aut(K/k)} \ne 0.$$
In the latter case, the induction assumption for $n$ applied to  $\mathcal{W}_H$  implies that
 $(\mathcal{W}_H)^{\Aut(K/k)} \ne 0$. Since $\mathcal{W}\supset \mathcal{W}_H$,
 we get $\mathcal{W}^{\Aut(K/k)}\ne 0$, which ends the proof.

{\bf Step 3} We have
$$\mathbf{T}_{k,K}(\tilde{\mathcal{W}}_0)=\tilde{\mathcal{W}}_0\otimes_k K\subset \mathcal{W}.$$
This implies that
$$n_0=\dim_k(\tilde{\mathcal{W}}_0) \le \dim_K(W)=n.$$
The assertion of our Lemma actually means that the equality holds.  By Step 2, $n_0>0$.
Suppose that $n_0<n$ and choose in $\mathcal{V}_0$ a $(m-n_0)$-dimensional 
$k$-vector subspace $\mathcal{U}_0$ such that $\mathcal{U}_0\cap \tilde{\mathcal{W}}_0=\{0\}$ (i.e., $\mathcal{V}_0=\mathcal{W}_0\oplus \mathcal{U}_0$). Let us consider the $(m-n_0)$-dimensional $K$-vector subspace
$$\mathcal{U}=\mathbf{T}_{k,K}(\mathcal{U}_0)=\mathcal{U}_0\otimes_k K\subset \mathcal{V}.$$
Clearly,
$$\mathcal{V}=\mathbf{T}_{k,K}(\tilde{\mathcal{W}}_0)\oplus \mathbf{T}_{k,K}(\tilde{\mathcal{U}}_0)=
\mathbf{T}_{k,K}(\tilde{\mathcal{W}}_0)\oplus\mathcal{U},$$
and therefore
$$\mathcal{U}\cap \mathbf{T}_{k,K}(\tilde{\mathcal{W}}_0)=\{0\}.$$
Dimension arguments imply that
$\mathcal{U}_1:=\mathcal{U}\cap  \mathcal{W}$ is a nonzero $\Aut(K/k)$-invariant $K$-vector subspace of $\mathcal{V}$.
By Step 2,  the subspace $\tilde{\mathcal{U}}_1:=  \mathcal{U}_1^{\Aut(K/k)} \ne \{0\}$; on the other hand, $\tilde{\mathcal{U}}_1$ obviously lies in $\mathcal{W}^{\Aut(K/k)}$
but meets the latter only at $\{0\}$. The obtained contradiction proves that $n_0=n$, which ends the proof.
\end{proof}

\begin{rem}
\label{dualKk}
Let us consider the dual vector spaces
$$\mathcal{V}_0^{*}=\Hom_k(\mathcal{V}_0,k), \quad \mathcal{V}^{*}=\Hom_K(\mathcal{V},K).$$
 Obviously, the restriction map
$$\mathrm{res}_{K,k}:\mathcal{V}^{*}=\Hom_K(\mathcal{V},K)=\Hom_K(\mathcal{V}_0\otimes_k K,K)  \to \Hom_k(\mathcal{V}_0,K),  \phi \mapsto \{v_0 \mapsto \phi(v_0\otimes 1)\}$$
is a $\Aut(K/k)$-equivariant  isomorphism of $K$-vector spaces where the actions of   $\Aut(K/k)$ are defined as follows.
$$\sigma: \phi \mapsto \sigma \circ \phi \circ \sigma^{-1} \quad \forall \phi \in \Hom_K(\mathcal{V},K),$$
$$\sigma: \phi_0 \mapsto \{v_0 \mapsto \sigma( \phi_0(v_0)) \quad \forall \phi_0 \in \Hom_k(\mathcal{V}_0,K)$$
for all $\sigma \in \Aut(K/k)$. As usual, we have
$$\sigma(\phi)(\sigma(v))=\sigma(\phi(v)) \quad \forall v\in \mathcal{V}, \phi \in \mathcal{V}^{*}, \sigma \in \Aut(K/k).$$
\end{rem}

\begin{sect}
\label{LieMod}
What is discussed in this section (and in Theorem \ref{mainLie} below) is pretty well known in the case of $k=\BR$ and $K=\BC$, see \cite{On}.

Let $\mathfrak{u}$ be a Lie $k$-algebra of finite dimension and
$$\bar{\mathfrak{u}}:=\mathfrak{u}\otimes_k K$$
the corresponding  finite-dimensional Lie $K$-algebra. Let
$$\rho: \mathfrak{u} \to  \End_K(\mathcal{V})$$
be a homomorphism of Lie $k$-algebras. Extending $\rho$ by $K$-linearity, we get the homomorphism of Lie $K$-algebras
$$\bar{\rho}: \bar{\mathfrak{u}} \to  \End_K(\mathcal{V}),$$
which coincides with $\rho$ on 
$$\mathfrak{u}=\mathfrak{u}\otimes 1\subset \mathfrak{u}\otimes_k K=\bar{\mathfrak{u}}.$$
  Thus $\bar{\rho}$ endows $\mathcal{V}$ with the structure 
of a  $\bar{\mathfrak{u}}$-module.

If $\sigma \in \Aut(K/k)$ then we may define the composition
$${^{\sigma}\rho}: \mathfrak{u}\overset{\rho}{\to} \End_K(\mathcal{V}) \overset{^{\sigma}\mathrm{id}}{\to}  \End_K({^{\sigma}\mathcal{V}}),$$
which is a  homomorphism of $k$-Lie algebras.  Then the corresponding homomorphism of Lie  $K$-algebras
$$\overline{^{\sigma}\rho}: \bar{\mathfrak{u}} \to  \End_K({^{\sigma}\mathcal{V}}),$$
provides ${^{\sigma}\mathcal{V}}$ with the  structure
of a   $\bar{\mathfrak{u}}$-module.

\begin{rem}
\label{invarW}
Let $\mathcal{W}$ be a $K$-vector subspace  of $\mathcal{V}$. 
Clearly,   $\mathcal{W}$ is $\mathfrak{u}$-invariant if and only if it is $\bar{\mathfrak{u}}$-invariant. It follows easily that  $\mathcal{W}$
is a $\bar{\mathfrak{u}}$-submodule of $\mathcal{V}$  if and only if it is a $\bar{\mathfrak{u}}$-submodule of ${^{\sigma}\mathcal{V}}$.
This implies that
the  $\bar{\mathfrak{u}}$-module $\mathcal{V}$ is {\sl simple} if and only if  the  $\bar{\mathfrak{u}}$-module ${^{\sigma}\mathcal{V}}$ is {\sl simple}.
\end{rem}
\end{sect}

\begin{sect}
\label{conjAut}
Let $\mathcal{V}_0$ be a finite dimensional $k$-vector space endowed with a homomorphism of $k$-Lie algebras
$$\rho_0: \mathfrak{u} \to \End_k(\mathcal{V}_0)$$
that endowed $\mathcal{V}_0$  with the structure of a $\mathfrak{u}$-module. Let us consider the $K$-vector space
$\mathcal{V}:=\mathcal{V}_0\otimes_k K$ and the obvious homomorphism of $k$-Lie algebras
$$\rho_0\otimes 1:\mathfrak{u}= \mathfrak{u}  \otimes 1 \to \End_k(\mathcal{V}_0)\otimes_k K =\End_K(\mathcal{V}_0\otimes K)=\End_K(\mathcal{V}).$$
obtained from $\rho_0$ by extension of scalars.

Let $\mathcal{W}$ be a  $\mathfrak{u}$-invariant $K$-vector subspace of $\mathcal{V}$. If $\sigma \in \Aut(K/k)$ then obviously $\sigma(\mathcal{W})$ is also a  
$\mathfrak{u}$-invariant $K$-vector subspace of $\mathcal{V}$. Clearly, both $\mathcal{W}$ and $\sigma(\mathcal{W})$ carry the natural structure of modules over the Lie $K$-algebra
 $$\bar{\mathfrak{u}}= \mathfrak{u}\otimes_k K.$$
 We will need the following assertion.
\end{sect}

\begin{prop}
\label{sigmaW}
The $\bar{\mathfrak{u}}$-modules $\sigma(\mathcal{W})$  and ${^{\sigma}\mathcal{W}}$  are isomorphic.
\end{prop}

\begin{proof}

It suffices to check that the $\mathfrak{u}$-modules $\sigma(\mathcal{W})$  and ${^{\sigma}\mathcal{W}}$  are isomorphic.
Let us consider the $k$-linear isomorphism
$$\Pi: \sigma(\mathcal{W}) \to   {^{\sigma}\mathcal{W}}, \quad \sigma(w) \mapsto w=\sigma^{-1}( \sigma(w)) \ \forall w \in \mathcal{W}.$$
Actually, $\Pi$ is $K$-linear, because  for all $a \in K, w\in \mathcal{W}$ the vector
$$\sigma^{-1}(a \sigma(w))=\sigma^{-1}(a)w \in \mathcal{W}$$
(recall that in ${^{\sigma}\mathcal{W}}$ multiplication by $a$ is defined as multiplication by $\sigma^{-1}(a)$).
Clearly, the actions of  $\mathfrak{u}$ and $\Aut(K/k)$ on $\mathcal{V}$ do commute.  This implies that
$$\Pi\circ \sigma=\sigma\circ \Pi \quad \forall \sigma \in \Aut(K/k).$$
It follows that $\Pi$ is an isomorphism of $\mathfrak{u}$-modules, which ends the proof.

\end{proof}

\section{Representations of semisimple Lie algebras and highest weights}
\label{RepsKk}
Throughout this section, $K$ is an {\sl algebraically closed} field of characteristic $0$ that contains a subfield $k$. We will need the following assertion.
\begin{thm}
\label{Conrad319}
Let $F$ be a subfield of $K$. Then every automorphism of $F$ can be extended to an automorphism of $K$.
\end{thm}

\begin{proof}
This is a special case (in characteristic 0) of Theorem 3.19  in \cite{ConradK}.
\end{proof}

 In order to be able to apply results of Section \ref{semiLin} to $K/k$, we need the following  assertion.

\begin{prop}
\label{AutKk}
$K^{\Aut(K/k)}=k$.
\end{prop}

\begin{proof}
Let $c \in K\setminus k$.  Thanks to Theorem \ref{Conrad319}, in order to prove that $c$ is {\sl not}  $\Aut(K/k)$-invariant,
it suffices to find an intermediate field $F$ such that
$$k \subset F \subset K, \quad c \in F,$$
and an automorphism $\sigma \in \Aut(F/k)$ such that 
$$\sigma(c) \ne c.$$

If $c$ is an algebraic over $k$ then it is a root of a certain monic irreducible polynomial $f(t) \in k[t]$. Let $F\subset K$ be the splitting field of $f(t)$ over $k$. Then 
$c \in F$ and $F/k$ 
is a finite {\sl Galois extension}. By Galois theory, 
$$F^{\Aut(F/k)}=k.$$
  Since $c \not\in k$,  there exists an automorphism $\sigma \in \Aut(F/k)$ such that 
$\sigma(c) \ne c$.

If $c$ is transcendental over $k$ then  the subfield $k(c)$ of $K$ generated by $c$ over $k$ is canonically isomorphic to the field $k(t)$ of rational  functions in an independent variable $t$
over the field of constants $k$.  Namely, there is a $k$-linear field isomorphism
$$\mathrm{ev}_c: k(t) \to k(c), \  \frac{u(t)}{v(t)} \mapsto  \frac{u(c)}{v(c)}$$
for any polynomials $u(t),v(t) \in k[t]$ with $v(t) \ne 0$. For example, $\mathrm{ev}_c$ sends the rational functions $t=\frac{t}{1}$ and 
 $t+1=\frac{t+1}{1}$ to $c$ and $c+1$ respectively.
 
 Now let us put
$F=k(c)$ and define its $k$-linear automorphism $\sigma$ by
$$\sigma: \frac{u(c)}{v(c)} \mapsto \frac{u(c+1)}{v(c+1)}.$$
Then $\sigma(c)=c+1 \ne c$.

This ends the proof.

\end{proof}

Recall that $K$ is {\sl algebraically closed} (e.g., $K$ is an algebraic closure of $k$).
Let $\mathfrak{g}$ be a nonzero semisimple finite-dimensional Lie algebra over $k$  of rank $l$ and consider the corresponding 
semisimple finite-dimensional Lie algebra 
$$\bar{\mathfrak{g}}:=\mathfrak{g}\otimes_k K$$
over $K$. If $\mathfrak{h}$ is a Cartan subalgebra of $\mathfrak{g}$ then 
$$\dim_k(\mathfrak{h})=l.$$
We write
$$\bar{\mathfrak{h}}:=\mathfrak{h}\otimes_k K\subset \mathfrak{g}\otimes_k K=\bar{\mathfrak{g}}$$
for the corresponding Cartan subalgebra of $\bar{\mathfrak{g}}$; we have
$$\dim_K(\bar{\mathfrak{h}})=l.$$
As usual, let us consider the dual $K$-vector space
$$\bar{\mathfrak{h}}^{*}:=\Hom_K(\bar{\mathfrak{h}},K)$$
of $K$-dimension $l$ endowed by the action of $\Aut(K/k)$  defined by the formula
$$\sigma \mapsto \{\phi: \mapsto \sigma\circ \phi \circ \sigma^{-1} \} \quad \forall \phi: \bar{\mathfrak{h}} \to K$$
and $\sigma \in \Aut(K/k)$.
As above, the restriction map
$$\mathrm{res}_{K,k}: \bar{\mathfrak{h}}^{*}:=\Hom_K(\bar{\mathfrak{h}},K) \to \Hom_k(\mathfrak{h},K)$$
is an isomorphism of $K$-vector spaces. (Here as above we identify $\mathfrak{h}$ with
$$\mathfrak{h}\otimes 1\subset \mathfrak{h}\otimes_k K=\bar{\mathfrak{h}}.)$$
The inverse map
$$\mathrm{res}_{K,k}^{-1}:\Hom_k(\mathfrak{h},K) \to  \Hom_K(\bar{\mathfrak{h}},K)=\Hom_K(\mathfrak{h}\otimes_k K,K),$$
is described explicitly by the formula
$$\mu \mapsto \{h \otimes a \mapsto a\cdot \mu(h)\} \quad \forall h \in \mathfrak{h}, a \in K.$$

Let $R\subset \bar{\mathfrak{h}}^{*}$ be the {\sl root system} of $(\bar{\mathfrak{g}},\bar{\mathfrak{h}})$
 \cite{Bourbaki78}. By definition, $R$ consists of all {\sl nonzero} $\alpha \in \bar{\mathfrak{h}}^{*}$ such that
 $$\bar{\mathfrak{g}}_{\alpha}:=\{x \in \bar{\mathfrak{g}}\mid [H,x]=\alpha(H)x \ \forall H \in 
 \bar{\mathfrak{h}}\} \ne \{0\}.$$
 Clearly,
 $$\bar{\mathfrak{g}}_{\alpha}=\{x \in \bar{\mathfrak{g}}\mid [H,x]=\alpha(H)x \ \forall H \in 
 \mathfrak{h}\}$$
 and therefore
 $$\sigma(\bar{\mathfrak{g}}_{\alpha})=\bar{\mathfrak{g}}_{\sigma(\alpha)} \quad \forall \sigma \in \Aut(K/k).$$
  It follows  that the subset $R$ of $\bar{\mathfrak{h}}^{*}$
 is $\Aut(K/k)$-invariant.
We write 
$$\mathrm{W}(R)\subset \Aut_K(\bar{\mathfrak{h}}^{*})$$ for the {\sl Weyl group} of the root system $R$. Notice that $\mathrm{W}(R)$ permutes elements of $R$.

Let us choose a {\sl basis} (a {\sl simple root system}) $B$ of $R$. The $l$-element  set $B$ is a basis of the $K$-vector space $\bar{\mathfrak{h}}^{*}$.
Every root $\alpha \in R$ is a linear combination of elements of $B$ with integer coefficients; in addition, the nonzero coefficients are either all positive
or all negative. (Actually, these properties characterize a {\sl basis} of $R$.) This implies the equality of abelian subgroups
\begin{equation}
\label{RlB}
\BZ\cdot R:=\sum_{\alpha\in R}\BZ\cdot \alpha=\sum_{\beta\in B}\BZ\cdot \beta;
\end{equation}
$\BZ\cdot R$ is a free abelian group of rank $l$ that is a $\mathrm{W}(R)$-invariant subgroup of $\bar{\mathfrak{h}}^{*}$.

The set $B$ does {\sl not} have to be $\Aut(K/k)$-invariant.  
However, if $\sigma \in \Aut(K/k)$ then 
$\sigma(B)$ is a basis of $R$ as well.  
Since the Weyl group $\mathrm{W}(R)$ acts transitively on the set of all simple root systems of $R$, there is $w_{\sigma}\in \mathrm{W}(R)$ such that
$$w_{\sigma}(\sigma(B))=B$$
(compare with \cite[p. 203]{Tits}). In particular,  
$$s_{\sigma}:=w_{\sigma}\circ \sigma \in \Aut_K(\bar{\mathfrak{h}}^{*})$$
permutes elements of $B$. It is also clear that $s_{\sigma}$ permutes elements of $R$. Hence,
$\BZ\cdot R$ is $s_{\sigma}$-invariant.
\begin{sect}
\label{highest}
Throughout this subsection we use the notation and constructions of  Subsection \ref{LieMod} applied to
$$\mathfrak{u}=\mathfrak{g},  \quad \bar{\mathfrak{u}}=\bar{\mathfrak{g}}.$$
Let $\mathcal{V}$ be a nonzero finite-dimensional vector space over $K$ endowed by the homomorphism of  Lie $K$-algebras
\begin{equation}
\label{rhoK}
\bar{\mathfrak{g}} \to  \End_K(\mathcal{V}),
\end{equation}
which may be viewed (in the notation of Subsection \ref{LieMod}) as $\bar{\rho}$ where
$$\rho: \mathfrak{g} \to \End_K(\mathcal{V})$$
is the restriction of the homomorphism \eqref{rhoK} to $\mathfrak{g}=\mathfrak{g}\otimes 1$.
The homomorphism $\bar{\rho}$ that appeared in \eqref{rhoK} provides $\mathcal{V}$ with the structure of a $\bar{\mathfrak{g}}$-module.
Let us assume that this module is {\sl simple}.

Let us consider the set
$$\mathrm{Supp}(\mathcal{V})\subset \bar{\mathfrak{h}}^{*}$$
 of weights of the  $\bar{\mathfrak{g}}$-module $\mathcal{V}$, i.e.,
$\mu \in \bar{\mathfrak{h}}^{*}$ lies in $\mathrm{Supp}(\mathcal{V})$ if and only if the {\sl weight subspace}
$$\mathcal{V}_{\mu}:=\{v \in \mathcal{V}\mid \rho(H)(x)=\mu(H)v \quad \forall H \in \bar{\mathfrak{h}}\} \ne \{0\}.$$
Then (see \cite{Bourbaki78})
\begin{equation}
\label{SuppQR}
\mathrm{Supp}(\mathcal{V})\subset \BQ\cdot R:=\sum_{\alpha\in R}\BQ\cdot \alpha=:\sum_{\beta\in B}\BQ\cdot \beta \subset \bar{\mathfrak{h}}^{*},
\end{equation}
and  there exists the {\sl highest weigh}t $\lambda$ of the $\bar{\mathfrak{g}}$-module $\mathcal{V}$ that enjoys the following properties.

\begin{itemize}
\item[(i)] $\lambda \in \mathrm{Supp}(\mathcal{V})$.

\item[(ii)]
If $\mu \in \mathrm{Supp}(\mathcal{V})$ then $\lambda-\mu$ is a linear combination of elements of $B$ with nonnegative integer coefficients.
\end{itemize}
\begin{rem}
It is well known that:
\begin{itemize}
\item[(i)]
$$\mathrm{Supp}(\mathcal{V})\subset \sum_{\beta\in B}\BQ\cdot \beta=\BQ\cdot R.$$
\item[(ii)]
The subset $\mathrm{Supp}(\mathcal{V})$ is $\mathrm{W}(R)$-invariant.
\end{itemize}
\end{rem}

\begin{rem}
\label{BRQ}
It follows from the $\mathrm{W}(R)$-invariance of $\BZ\cdot R$ (see \eqref{RlB}) that the $l$-dimensional $\BQ$-vector (sub)space 
$\BQ\cdot R$ is $\mathrm{W}(R)$-invariant.
\end{rem}
\end{sect}
Recall (Subsection \ref{LieMod})  that one may attach to each $\sigma \in \Aut(K/k)$ the  homomorphism of Lie $K$-algebras 
$$\overline{^{\sigma}\rho}: \bar{\mathfrak{g}} \to  \End_K({^{\sigma}\mathcal{V}}),$$
and the corresponding  $\bar{\mathfrak{g}}$-module ${^{\sigma}\mathcal{V}}$ is simple.
\begin{thm}
\label{mainLie}
Suppose that  $\lambda$ is the dominant weight of a simple $\bar{\mathfrak{g}}$-module $\mathcal{V}$ of finite dimension.
If $\sigma \in \Aut(K/k)$ then $s_{\sigma^{-1}}(\lambda)$ is  the dominant weight of the simple $\bar{\mathfrak{g}}$-module ${^{\sigma}\mathcal{V}}$.
\end{thm}

\begin{proof}
First, notice that
$$\mathrm{Supp}({^{\sigma}\mathcal{V}})=\sigma(\mathrm{Supp}(\mathcal{V})).$$
Indeed, for any $$H \in \mathfrak{h}\subset \bar{\mathfrak{h}},$$
 the spectrum of the diagonalizable operator $\rho(H)$ in $\mathcal{V}$
is the collection 
$$\{\mu(H)\mid \mu \in \mathrm{Supp}(\mathcal{V})\} \ \text{ (with multiplicities)}.$$ In light of  
Lemma \ref{EndVsigma},
the spectrum of the  diagonalizable operator   ${^{\sigma}\mathrm{id}}\circ \rho(H)$ in ${^{\sigma}\mathcal{V}}$
is the collection 
$$\{\sigma(\mu(H))\mid \mu \in \mathrm{Supp}(\mathcal{V})\} \ \text{(with multiplicities)}.$$
More precisely, let $n=\dim(\mathcal{V})$ and $\{e_1, \dots,  e_n\}$ be a common (weight) eigenbasis of all elements of $\bar{\mathfrak{h}}$ in $\mathcal{V}$, i.e.,
for each index $i \in \{1, \dots, n\}$ there is a weight 
$$\mu_i \in \mathrm{Supp}(\mathcal{V})\in \bar{\mathfrak{h}}^{*}$$
such that
$$\rho(H)(e_i)=\mu_i(H)e_i \ \forall H \in \bar{\mathfrak{h}}.$$
(Clearly, the collection $\{\mu_1, \dots, \mu_n\}$ coincides with $\mathrm{Supp}(\mathcal{V})$.)
In light of  Remark \ref{sigmaM},
$\{e_1, \dots,  e_n\}$ is a basis of ${^{\sigma}\mathcal{V}}$, and 
if $H \in \mathfrak{h}$, then $H=\sigma^{-1}H$ and
\begin{equation}
\label{weightSigma}
{^{\sigma}\mathrm{id}}\circ \rho(H)(e_i)=\sigma(\mu_i(H)) e_i=\sigma(\mu_i(\sigma^{-1}H))e_i =(\sigma(\mu_i))(H)e_i.
\end{equation}
Since $\bar{\mathfrak{h}}=\mathfrak{h}\otimes_k K$, we conclude that
\begin{equation}
\label{weightSigmaBar}
{\overline{^{\sigma}\rho}}(H)(e_i)=(\sigma(\mu_i))(H)e_i \ \forall H \in \bar{\mathfrak{h}}.
\end{equation}
In other words,
$$\mathrm{Supp}({^{\sigma}\mathcal{V}})=\{\sigma\circ \mu_i \mid i=1, \dots, n\}=\{\sigma\circ \mu\mid \mu \in \mathrm{Supp}(\mathcal{V})\}.$$

Second, the $\mathrm{W}(R)$-invariance of $\mathrm{Supp}({^{\sigma}\mathcal{V}})$  implies that
$$\mathrm{Supp}({^{\sigma}\mathcal{V}})=w_{\sigma}\circ \sigma(\mathrm{Supp}(\mathcal{V}))=(w_{\sigma}\circ \sigma)(\mathrm{Supp}(\mathcal{V}))
=s_{\sigma}(\mathrm{Supp}(\mathcal{V})).$$
It follows that $\mathrm{Supp}({^{\sigma}\mathcal{V}})$ contains $s_{\sigma}(\lambda)$, and all the other weights in 
$\mathrm{Supp}({^{\sigma}\mathcal{V}})$ are of the form $s_{\sigma}(\mu)$ where $\lambda-\mu$ is 
is a linear combination of elements of $B$ with nonnegative integer coefficients. Since $s_{\sigma}$ permutes elements of $B$, the difference
$s_{\sigma}(\lambda)-s_{\sigma}(\mu)$ is also a linear combination of elements of $B$ with nonnegative integer coefficients. It follows that
$s_{\sigma}(\lambda)$ is  the dominant weight of the simple $\mathfrak{g}$-module ${^{\sigma}\mathcal{V}}$.

\end{proof}

\section{$2$-Simple complex tori without nontrivial endomorphisms}
\label{noEnd}

\begin{thm}
\label{SlSoSp}
Suppose that $T$ is a $2$-simple complex torus of dimension $g\ge 3$ with $\End^0(T)=\BQ$.  Assume also
that $g \ne 10$ and $2g$ is not a power (e.g., $g$ is odd) .
 Then $\mathrm{Hdg}(T)$ enjoys one of the following properties.

\begin{itemize}
\item[(i)]
$\mathrm{Hdg}(T)=\SL(\Lambda_{\BQ})$;

\item[(ii)] There exists a nondegenerate  symmetric $\BQ$-bilinear form
$$\Lambda_{\BQ}\times \Lambda_{\BQ} \to \BQ$$
such that $\mathrm{Hdg}(T)$ coincides with the corresponding special orthoginal group
$\mathrm{SO}(\Lambda_{\BQ})$.
\item[(iii)] There exists a nondegenerate alternating $\BQ$-bilinear form
$$\Lambda_{\BQ}\times \Lambda_{\BQ} \to \BQ$$
such that $\mathrm{Hdg}(T)$ coincides with the corresponding symplectic group
$\mathrm{Sp}(\Lambda_{\BQ})$.
\end{itemize}
\end{thm}

\begin{proof}

It follows from  Corollary \ref{BorTabs} that 
$\mathrm{hdg}_{T,\BC}$ is a complex {\sl simple} classical Lie algebra, whose natural  faithful representation in $\Lambda_{\BC}$ has a minuscule weight as the highest weight, thanks to Theorem \ref{deligneT}.  Since $2g=\dim_{\BC}(\Lambda_{\BC})$ is not a power of $2$, one should exclude the cases when either $\mathrm{hdg}_{T,\BC}$ is of type ${\sf{B}}_l$, or $\mathrm{hdg}_{T,\BC}$ is of type ${\sf{D}}_l$ and $\Lambda_{\BC}$ is one of its two semi-spinorial representations. Let us list the remaining cases.

\begin{itemize}
\item[(i)]
$\mathrm{hdg}_{T,\BC}$ is of type ${\sf{C}}_g$  or ${\sf{D}}_g$, and there is a nondegenerate alternating or symmetric bilinear form on $\Lambda_{\BC}$ such that $\mathrm{hdg}_{T,\BC}$ coincides with the corresponding symplectic Lie algebra $\mathfrak{sp}(\Lambda_{\BC})$ or the corresponding orthogonal Lie algebra $\mathfrak{so}(\Lambda_{\BC})$.
\item[(ii)]
$\mathrm{hdg}_{T,\BC}$ is of type ${\sf{A}}_l$, i.e., $\mathrm{hdg}_{T,\BC}$ may be identified with the Lie algebra $\mathfrak{sl}(W)$ of a $(l+1)$-dimensional complex vector space $W$ in such a way that the $\mathfrak{sl}(W)$-module $\Lambda_{\BC}$ is isomorphic to the $j$th exterior power
$\wedge_{\BC}^j(W)$ of $W$ for some integer $j$ with $1 \le j \le l$.  We may  assume that $1<j<l$.  
\end{itemize}

Let us handle the case (i).  In this situation the $\mathrm{hdg}_{T,\BC}$-module $\Lambda_{\BC}$ is self-dual, which implies that there is a non-zero homomorphism
between the $\mathrm{hdg}_{T,\BC}$-module $\Lambda_{\BC}$ and its dual. This, in turn, implies that there is a non-zero homomorphism
between the $\mathrm{hdg}_{T,\BQ}$-module $\Lambda_{\BQ}$ and its dual. Now the  simplicity of the  $\mathrm{hdg}_{T}$-module $\Lambda_{\BQ}$ implies
that $\Lambda_{\BQ}$ and its dual are isomorphic, i.e., there is a nondegenerate $\mathrm{hdg}_{T}$-invariant bilinear form
$$\Lambda_{\BQ} \times \Lambda_{\BQ} \to \BQ.$$
The absolutely simplicity  of the $\mathrm{hdg}_{T,\BQ}$-module $\Lambda_{\BQ}$ implies that this form is unique (up to multiplication by a non-zero rational number)
and, therefore, is 
 alternating if $\mathrm{hdg}_{T}$ is of type $\sf{C}_g$ or symmetric if $\mathrm{hdg}_{T}$ is of type $\sf{D}_g$. Now the dimension arguments imply that
$\mathrm{hdg}_{T}=\mathfrak{sp}(\Lambda_{\BC})$ in the former case and   $\mathrm{hdg}_{T}=\mathrm{so}(\Lambda_{\BC})$ in the latter case.

Let us handle the case (ii).    
We know that $T=V/\Lambda$ where the $\mathrm{hdg}_{T,\BC}=\sll(W)$-module $V_{\BC}$ is isomorphic to 
$\wedge_{\BC}^j(W)$. 

If $l=1$ then the inequality  $1<j<l=1$ implies that this case does not occur.

If $l=2$ then the inequality  $1<j<l=2$ implies that $j=1$ and $V_{\BC}$ is isomorphic to $W$, which is a $3$-dimensional  complex vector space.
Since $3$ is an odd integer and the $\BC$-dimension of  $V_{\BC}$ is even, this case also does not occur.

If $l=3$ then the inequality  $1<j<l=3$ implies that $j=2$ and $V_{\BC}$ is isomorphic to $\wedge^2_{\BC}W$ where $W$ is a $4$-dimensional complex vector space
and $\wedge^2_{\BC}W$ is an irreducible  $6$-dimensional orthogonal representation of the Lie algebra $\sll(W)$.  This implies that  the representation
of $\mathrm{hdg}_{T,\BC}$ in the $6$-dimensional complex vector space $V_{\BC}$ is orthogonal and irreducible. It follows that 
$$\dim_{\BQ}(\Lambda_{\BQ})=\dim_{\BC}(W)=6$$ and 
$\mathrm{hdg}_{T}$ is a $\BQ$-Lie subalgebra $\so(\Lambda_{\BQ})$ of the corresponding  special orthogonal group $\mathrm{SO}(\Lambda_{\BQ})$.
It follows  that
the representation
of $\mathrm{hdg}_{T}$ in the $6$-dimensional $\BQ$-vector space $\Lambda_{\BQ}$ is orthogonal and irreducible, and therefore 
$$\dim_{\BQ}(\mathrm{hdg}_{T}) \le \dim_{\BQ}(\so(\Lambda_{\BQ})=15.$$
However,
$$\dim_{\BQ}(\mathrm{hdg}_{T}) =\dim_{\BC}(\mathrm{hdg}_{T,\BC}) =\dim_{\BC}(\sll(W))=15.$$
This implies that $\mathrm{hdg}_{T}$ coincides with $\so(\Lambda_{\BQ})$, i.e., $\mathrm{Hdg}_{T}$
coincides with  $\mathrm{SO}(\Lambda_{\BQ})$.

So, we may and will assume that $l>3$.
Then there is an element $u\in \sll(W)$ that acts on $V_{\BC}$ as $J$.  Since $J$ is a nonzero  semisimple linear operator in $V_{\BC}$,  the element $u$ is also a semisimple (i.e., diagonalizable) 
nonzero linear operator
in $W$.  Let $\{e_1, \dots, e_{l+1}\}$ be an eigenbasis of $W$ and   $\{z_1, \dots, z_{l+1}\}\subset \BC$ be the corresponding eigenvalues of $u$, i.e,
$$u( e_{i})=z_{i}e_{i} \ i=1, \dots, l+1$$
and the trace of $u$ is
$$\sum_{i=1}^{l+1} z_{i}=0.$$
This implies that $u$ has at least two distinct eigenvalues.

 Then the collections of eigenvalues of $J$ in $V_{\BC}\cong \wedge^j_{\BC}(W)$ listed with
multiplicities coincides with 
$$\{t_A:=\sum_{i\in A}z_{i}\}_A$$
 where $A$ runs through all $j$-element subsets $A$ of $\{1, \dots, l+1\}$. On the other hand, we know that
the spectrum of $J$ in $V_{\BC}$ consists of two eigenvalues $\mathbf{i}$ and $-\mathbf{i}$, whose multiplicities coincide. It follows almost immediately that 
$u$ has precisely two (distinct) eigenvalues, say, $\mathbf{a}$ and $\mathbf{b}$, and none of them is $0$. 
Indeed, suppose that  the spectrum of $u$ contains (at least) three eigenvalues, say, $\mathbf{a},\mathbf{b},\mathbf{c}$. 

Reordering the eigenbasis if necessary, we may assume that
$$z_1=\mathbf{a}, \ z_2=\mathbf{b}, \ z_3=\mathbf{c}.$$
Let $B$ be any $(j-2)$-element subset of $\{4, \dots, l+1\}$
Let us consider three  distinct $j$-element subsets
$$A_1=\{2,3\}\cup B, \ A_2=\{1,3\}\cup B, \ A_3=\{1,2\}\cup B$$
of $\{1, \dots, l+1\}$. If we put 
$$C:=\{1,2,3\}\cup B\subset \{1, \dots, l+1\}, \quad c:=\sum_{i\in C}z_{i}\in \BC$$
then we get  three distinct eigenvalues
$$t_{A_1}=c-\mathbf{a}, \ t_{A_2}=c-\mathbf{b}, \ t_{A_3}=c-\mathbf{c}$$
of $J$, which do not exist. This proves that the spectrum of $u$ consists of precisely two eigenvalues, say,
$\mathbf{a}, \mathbf{b} \in \BC$. Since the trace of {\sl nonzero} $u$ is $0$, both $\mathbf{a}$ and $\mathbf{b}$ are {\sl not} zero.

Let $p$ be the multiplicity of the eigenvalue $\mathbf{a}$ and $q$ the multiplicity of the eigenvalue $\mathbf{b}$.
Both $p$ and $q$ are positive integers, whose sum 
$$p+q=l+1>3+1=4.$$
 Since $u$ is traceless,
$$p \mathbf{a}+q \mathbf{b}=0.$$
I claim that either $p=1$ or $q=1$. Indeed, suppose that 
$$p\ge 2, \ q\ge 2.$$
Since $p+q>4$, we may assume that $p \ge 3$.
Notice also that all three complex numbers
$$2\mathbf{a}, \ 2\mathbf{b},\  \mathbf{a}+\mathbf{b}$$
are distinct.
Reordering the eigenbasis if necessary, we may assume that
$$z_1=\mathbf{a}, \ z_2=\mathbf{a}, z_3=\mathbf{a}, \ z_l=\mathbf{b}, \ z_{l+1}=\mathbf{b}$$
(recall that $l+1>4$). 

  Let $\mathbf{B}$ be a $(j-2)$-element subset of the $(l-3)$-element subset  $\{3,4, \dots, l-1\}$ and $b:=\sum_{i\in \mathbf{B}} z_{i}$.
  Let us consider three  distinct $j$-element subsets
$$A_1=\{1,2\}\cup B, \ A_2=\{l,l+1\}\cup B, \ A_3=\{1,l\}\cup B$$
of $\{1, \dots, l+1\}$. 
Then we get  three {\sl distinct} eigenvalues
$$t_{A_1}=b+2\mathbf{a}, \ t_{A_2}=b+2\mathbf{b}, \ t_{A_3}=b+(\mathbf{a}+\mathbf{b})$$
of $J$, which could not be the case. The obtained contradiction proves that either $p=1$ or $q=1$.

Without loss of generality we may assume that $p=1$. Reordering the eigenbasis if necessary, we may assume that
$z_1=\mathbf{a}$ and all other $z_{i}=\mathbf{b}$ (for all $i>1$). It follows easily that the spectrum of $J$ consists of two eigenvalues,
namely, $j\mathbf{b}$ of multiplicity $\binom{l}{j}$ and $\mathbf{a}+(j-1)\mathbf{b}$ of multiplicity $\binom{l}{j-1}$. It follows that
$$\binom{l}{j}=\binom{l}{j-1},$$
i.e.,
$$\frac{l-j+1}{j}=1, \quad l-j+1=j, \quad l=2j-1.$$
It remains to put $m=j$ and we get that 
$$l=2m-1, \ j=m, \ 2g=\binom{2m}{m}.$$
Since $2m-1=l>3$, we  get
\begin{equation}
\label{mB3}
m \ge 3.
\end{equation}

Now it is natural to look at   the  structure of the $\sll(W)$-module $\wedge^2_{\BC}\left(\wedge_{\BC}^m(W)\right)$. We are going to apply results of Section \ref{semiLin} with
$$k=\BQ, \ K=\BC,\  \Aut(K/k)=\Aut(\BC),$$
$$\mathfrak{g}=\mathrm{hdg}_T, \ \bar{\mathfrak{g}}=\mathrm{hdg}_{T,\BC}, \
\mathcal{V}=\wedge_{\BQ}^2(\wedge_{\BQ}^m \Lambda_{\BQ}), \
\bar{\mathcal{V}}=\wedge_{\BC}^2(\wedge_{\BC}^m \Lambda_{\BC}).$$

Let us fix a Cartan subalgebra $\mathfrak{h}$ of the simple Lie $\BQ$-algebra $\mathrm{hdg}_T$, which is an $l$-dimensional $\BQ$-vector space. Then
$$\bar{\mathfrak{h}}=\mathfrak{h}\otimes_{\BQ}\otimes\BC$$
is a Cartan subalgebra of  the complex simple Lie algebra $\mathrm{hdg}_{T,\BC}$ that is an $l$-dimensional complex vector space endowed with the natural semi-linear $\Aut(\BC)$-action;
its subalgebra of invariants coincides with $\mathfrak{h}\otimes 1=\mathfrak{h}$.

 As in Section \ref{semiLin}, let  us consider the 
 dual complex vector space
 $$\bar{\mathfrak{h}}^{*}=\Hom_{\BC}(\bar{\mathfrak{h}},\BC).$$ 

 Let $$R\subset \overline{\mathfrak{h}}^{*}$$
 be the  root system of $(\mathrm{hdg}_{T,\BC}, \overline{\mathfrak{h}})$.

 Let us choose a simple root system $B$ of simple roots (basis) of $R$ and let
 $ P_{++}(R)\subset \overline{\mathfrak{h}}^{*}$ be the corresponding semigroup of dominant weights  \cite{BourbakiR}.

If $\mu \in P_{++}(R)$ then we write $\mathbf{V}(\mu)$ for the simple $\mathrm{hdg}_{T,\BC}$-module with highest weight $\mu$ \cite{Bourbaki78}.
In particular, $\mathbf{V}(0)$ stands for the one-dimensional $\bar{\BQ}$-vector space $\bar{\BQ}$ with trivial (zero) action of  $\mathrm{hdg}_{T,\BC}$.
Then the $\binom{2m}{m}$-dimensional $\bar{\BQ}$-vector space
$$\bar{\Lambda}=\Lambda_{\BQ}\otimes_{\BQ}\bar{\BQ}$$
becomes a simple $\mathrm{hdg}_{T,\BC}$-module that is isomorphic to $\mathbf{V}(\bar{\omega}_m)$. Hereafter we use  the notation of Bourbaki (\cite[Tables]{BourbakiR},
\cite[Tables]{Bourbaki78}).
In particular, 
$$B=\{\alpha_1, \dots,\alpha_l\}=\{\alpha_1, \dots,\alpha_{2m-1}\}$$
 (see Root systems of type ${\sf{A}}_l$ in \cite[Tables]{BourbakiR}), and 
$\bar{\omega}_{i}$ is the dominant weight of a fundamental representation of dimension $\binom{2m}{i}$ (when $1 \le i \le l=2m-1$),
see \cite[Table 2]{Bourbaki78}.
In addition, we put 
$$\bar{\omega}_0:=0=:\bar{\omega}_{2m}.$$
Notice that the only nontrivial automorphism of $(R, B)$ is the involution
$$\alpha_i \to \alpha_{2m-i} \quad \forall i=1, \dots, 2m-1=l.$$
Hence, each dominant weight $\bar{\omega}_{m+i}+\bar{\omega}_{m-i}$ is  $\Aut(R, \mathbf{B})$-invariant for all $i=0, \dots, 2m$.

It follows from results of \cite[p. 140, Example 9a, last displayed formula]{Mc} (see also
\cite[Exercises 6.16 on p. 81 and 15.32 on p. 226]{FH}) that the $\bar{\mathfrak{g}}=\mathrm{hdg}_{T,\BC}$-module
$$\bar{\mathcal{V}}=\wedge^2_{\BC}(\mathbf{V}(\bar{\omega}_m))$$ is isomorphic to a direct sum
\begin{equation}
\label{partition}
\oplus_{i \ \text{odd}, \ 1 \le i\le m} \mathbf{V}(\bar{\omega}_{m+i}+\bar{\omega}_{m-i}).
\end{equation}
This implies that  the $\bar{\mathfrak{g}}$-module $\bar{\mathcal{V}}$
splits into  a direct sum of mutually non-isomorphic  simple $\bar{\mathfrak{g}}$-modules;
one of them is trivial if and only if $m$ is odd (one should take $i=m$ in order to get the summand $\mathbf{V}(0)$.)

Let $\mathcal{W}$ be a simple $\bar{\mathfrak{g}}$-submodule of $\bar{\mathcal{V}}$. Let $\lambda_{\mathcal{W}}$ be the highest weight of $\mathcal{W}$. We know that $\lambda_{\mathcal{W}}$ is $\Aut(R,B)$-invariant.  It follows from 
Theorem \ref{mainLie}  combined with Proposition \ref{sigmaW} that the simple $\bar{\mathfrak{g}}$-submodules 
$\mathcal{W}$ and $\sigma(\mathcal{W})$ have the same highest weight and therefore are isomorphic. This implies that
$$\sigma{\mathcal{W}}=\mathcal{W} \ \forall \sigma \in \Aut(\BC).$$
By Lemma \ref{subspace}, $\mathcal{W}$ is defined over $\BQ$, i.e., there is a $\BQ$-vector subspace $\mathbf{W}$ of $\mathcal{V}$ such that
$$\mathcal{W}=\mathbf{W} \otimes_{\BQ}\BC.$$
Clearly, such $\mathbf{W}$ is a simple $\mathrm{hdg}_T$-submodule of $\mathcal{V}$.  
It follows from \eqref{partition}  that the $\mathrm{hdg}_T$-module $\mathcal{V}$ splits into a direct sum
\begin{equation}
\label{partition0}
\oplus_{i \ \text{odd}, \ 1 \le i\le m} \mathbf{W}_i.
\end{equation}
of $\mathrm{hdg}_T$-modules such that 
$$\mathbf{V}(\bar{\omega}_{m+i}+\bar{\omega}_{m-i})\cong \mathbf{T}_{\BQ,\BC}(\mathbf{W}_i)=\mathbf{W}_i\otimes_{\BQ}\BC.$$
This implies that all $\mathbf{W}_i$ are mutually non-isomorphic simple  $\mathrm{hdg}_T$-modules. In adddition, one of them is trivial if and only if $m$ is odd. (Namely, if $m$ is odd then 
$\mathbf{W}_m$ is a trivial  $\mathrm{hdg}_T$-module of $\BQ$-dimension $1$.)

Thus, if $m$ is {\sl even},  then the $\mathrm{hdg}_T$-module $\mathcal{V}$ splits into a direct sum of $(m/2)$ simple modules, none of which is trivial. If $m$ is {\sl odd},  then the $\mathrm{hdg}_T$-module $\mathcal{V}$ splits into a direct sum of $(m+1)/2$ simple modules, and precisely one of them is trivial.  It follows that $\mathrm{hdg}_T$-module $\mathcal{V}$ is simple if and only if $m=2$.    Since  $m \ge 3$ \eqref{mB3}, we conclude that 
$\mathcal{V}$ is never simple. On the other hand, it's clear that 
$\mathcal{V}$ is a direct sum of a simple $\mathrm{hdg}_T$-module and a trivial one if and only if $m=3$.

Recall that we are actually interested in the dual $\mathrm{hdg}_T$-module
$$\mathrm{H}^2(T,\BQ)=\Hom_{\BQ}(\mathcal{V},\BQ).$$
By duality, the $\mathrm{hdg}_T$-module is never simple; it is a direct sum of a simple $\mathrm{hdg}_T$-module and a trivial one if and only if $m=3$. Now the 2-simplicity of $T$ implies that $m=3$ and therefore 
$$2g=\binom{2\cdot 3}{3}=20,$$
i.e., $g=10$.
\end{proof}

\end{document}